\documentclass[11pt]{article}

\usepackage{bm}%
\usepackage{graphicx}%
\usepackage{multirow}%
\usepackage{amsmath,amssymb,amsfonts}%
\usepackage{amsthm}%
\usepackage{mathrsfs}%
\usepackage[title]{appendix}%
\usepackage{xcolor}%
\usepackage{textcomp}%
\usepackage{manyfoot}%
\usepackage{booktabs}%
\usepackage{algorithm}%
\usepackage{algorithmicx}%
\usepackage{algpseudocode}%
\usepackage{listings}%

\newtheorem{theorem}{Theorem}[section]
\newtheorem{corollary}[theorem]{Corollary}
\newtheorem{lemma}[theorem]{Lemma}
\newtheorem{example}[theorem]{Example}
\newtheorem{proposition}[theorem]{Proposition}
\newtheorem{remark}[theorem]{Remark}
\newtheorem{definition}[theorem]{Definition}

\newcommand{\R}{\mathbb{R}}

\newcommand{\bn}{\mbox{\boldmath $n$}}
\newcommand{\bt}{\mbox{\boldmath $t$}}
\newcommand{\bs}{\mbox{\boldmath $s$}}
\newcommand{\ba}{\mbox{\boldmath $a$}}
\newcommand{\bb}{\mbox{\boldmath $b$}}
\newcommand{\bc}{\mbox{\boldmath $c$}}
\newcommand{\be}{\mbox{\boldmath $e$}}

\newcommand{\bx}{\mbox{\boldmath $x$}}
\newcommand{\by}{\mbox{\boldmath $y$}}

\newcommand{\bgamma}{\mbox{\boldmath $\gamma$}}
\newcommand{\bnu}{\mbox{\boldmath $\nu$}}
\newcommand{\bmu}{\mbox{\boldmath $\mu$}}
 
\begin{document}
 
\title{Hyperbolic generalized framed surfaces in hyperbolic 3-space}
\author{Donghe Pei, Masatomo Takahashi, Anjie Zhou}
\date{}
\maketitle

\abstract{Generalizing both hyperbolic framed surfaces and one-parameter families of hyperbolic framed curves, we introduce  the concept of hyperbolic generalized framed surfaces and establish their relations  in hyperbolic 3-space. We provide the necessary and sufficient conditions for a smooth surface to be a hyperbolic generalized framed base surface, followed by an analysis of the singularities  of hyperbolic generalized framed base surfaces. Additionally,   relations between hyperbolic generalized framed surfaces, generalized framed surfaces and lightcone framed surfaces are explored. As an application of hyperbolic generalized framed surfaces, we investigate the properties of horocyclic surfaces.
}
\renewcommand{\thefootnote}{\fnsymbol{footnote}}
\footnote[0]{2020 Mathematics Subject classification: 53A05, 57R45, 53B30,  58K05}
\footnote[0]{key words and phrases: hyperbolic generalized surface, horocyclic surface, singularity, hyperbolic 3-space}
\section{Introduction}
In differential geometry, singular surfaces have long been a central focus of research, and an increasing number of powerful tools have been developed to analyze them. In Euclidean space,   a widely used  tool is the framed surface introduced by Fukunaga and the second author   (cf. \cite{FT2019}). 
The framed surface  possesses  a smooth moving frame  at its singularities, which makes it possible to study the properties at these  points. Numerous articles have applied the theory of framed surfaces to various classes of surfaces, yielding many results (cf. \cite{DIT2023,EYA2022,FT2022,Y2022,YAT2021}). 
However, there are some surfaces that are  not framed surfaces, a typical example is the one-parameter family of framed curves,   and the relation between one-parameter families of framed curves and framed surfaces was discussed in \cite{FT2020}.  
To treat a broader class of singular surfaces, the second author and Yu introduced generalized framed surfaces in \cite{TY2024}. Generalized framed surfaces  include   framed surfaces  and also serve  as a generalization of the one-parameter families of framed curves. More singular surfaces, such as surfaces with corank one singularities, can be studied through it.

For the study of singular surfaces in different ambient spaces, we  introduced the lightcone framed surface, equipped with a lightcone frame, to study mixed-type surfaces with singularities in Lorentz–Minkowski 3-space  (cf. \cite{LPT2025}). The lightcone framed surface is closely related to   generalized framed surface in Euclidean 3-space, see Theorem 3.5 in \cite{LPT2025}. As a model of hyperbolic geometry, hyperbolic space has also been extensively studied (cf. \cite{I2009,IST2010,IT2008,ST2020,YC2025,YCW2024,ZYP2024,ZP}). In hyperbolic 3-space, hyperbolic framed curves have been used to investigate  various geometric objects (cf. \cite{YC2025,ZYP2024}).
The first and third authors  have   defined the hyperbolic framed surface in \cite{ZP}.  A curve in hyperbolic 3-space with curvature $ 1 $ and torsion $ 0 $ is known as a \textit{horocycle}. Treating the horocycle as a   ``straight line" yields horospherical geometry (cf. \cite{I2009}). A one-parameter family of horocycles is called a \textit{horocyclic surface}, which may not satisfy the definition of   hyperbolic framed base surface, as shown in  Proposition \ref{th4.2}. Similarly, a one-parameter family of hyperbolic framed curves is not necessarily a hyperbolic framed surface. These and other examples motivate the development of a more general notion of singular surfaces in hyperbolic 3-space.

This paper aims to  introduce the concept of hyperbolic generalized framed surfaces in hyperbolic 3-space, extending hyperbolic framed surfaces and one-parameter families of hyperbolic framed curves.
In Section \ref{S2}, we review
the theories of hyperbolic framed surfaces, one-parameter families of hyperbolic framed curves in hyperbolic 3-space, generalized framed surfaces in  Euclidean 3-space and lightcone framed surfaces in Lorentz-Minkowski 3-space.
We define hyperbolic generalized framed surfaces and  study their properties in Section \ref{S3}. We establish the relations  between hyperbolic generalized framed surfaces, hyperbolic framed surfaces and one-parameter families of hyperbolic framed curves (Propositions \ref{prop5} and \ref{prop6}). The relations between their base surfaces are also shown by considering the rotation of   frame (Propositions \ref{th11} and \ref{prop13}). Additionally, we provide  
necessary and sufficient conditions for a smooth surface to be a hyperbolic generalized framed base surface (Theorem \ref{th7}). The singularities of such base surfaces are analyzed using their basic invariants  and several examples are presented.
In Section \ref{S4}, we discuss the connections between hyperbolic generalized framed surfaces, generalized framed surfaces of Euclidean 3-space and lightcone framed surfaces of Lorentz-Minkowski 3-space.  As an application, Section \ref{S5} shows that horocyclic surfaces are hyperbolic generalized framed base surfaces and investigates their properties.

All maps and manifolds considered here are of class ${C^\infty}$ unless otherwise stated.
\paragraph{Acknowledgements} The first and third authors are partially supported by National
Natural Science Foundation of China (Grant No. 11671070). The second author
is partially supported by JSPS KAKENHI (Grant No. JP 24K06728). The third author is  grateful for the financial support provided by the CSC-MuroranIT Scholarship (Grant No. 202506620020).

\section{Preliminaries}\label{S2}
We quickly review the theories of  hyperbolic framed surfaces, one-parameter families of hyperbolic framed curves, generalized framed surfaces   and lightcone framed surfaces.

Let $\R_1^n $ denote   \textit{Lorentz-Minkowski n-space}   with the pseudo inner product $ \langle \bx_1,\bx_2\rangle=-x^1_{1}x^2_{1}+ \sum^{n}_{j=2}x^1_{j}x^2_{j}    $ and the  pseudo vector product $$\bx_1 \wedge \bx_2\wedge\dots\wedge\bx_{n-1}=\det\begin{pmatrix}-\be_1&\be_2&\dots&\be_n\\
	x^1_{1}&x^1_{2}&\dots   &x^1_{n}\\
	x^2_{1}&x^2_{2}&\dots  &x^2_{n}\\
	\vdots&\vdots& \ddots & \vdots \\
		x^{n-1}_{1}&x^{n-1}_{2}&\dots   &x^{n-1}_{n}
	\end{pmatrix},$$ 
   where $\bx_i =(x^i_{1},x^i_{2},\dots,x^i_{n})\in\R_1^n$  ($i=1,2,\dots,n-1$) and $\left\{\be_1,\be_2,\dots,\be_n\right\}$ is the canonical basis of $\R_1^n$. For any $\bx\in\R_1^n$, one easily checks that
$  \langle \bx,\bx_1 \wedge \bx_2 \wedge\dots\wedge\bx_{n-1}\rangle=\det(\bx,\bx_1,\bx_2,\dots,\bx_{n-1}), $
so that $\bx_1 \wedge \bx_2 \wedge\dots\wedge\bx_{n-1}$ is pseudo-orthogonal to any $\bx_i$.    A vector $\bx\in\R_1^n\backslash \left\{\bm{0}\right\} $ is $\it{spacelike}$, $\it{lightlike}$ or $\it{timelike}$ if $\langle \bx,\bx\rangle\textgreater0$, $\langle \bx,\bx\rangle=0$ or $\langle \bx,\bx\rangle\textless0$, respectively.
The norm of vector $\bx$ is given by 
$\Vert \bx\Vert =\sqrt{\vert \langle \bx,\bx\rangle \vert}. $

We define the \textit{hyperbolic 3-space} as
$$H^3=\left\{ \bx\in\R_1^4 \:|\: \langle \bx,\bx\rangle=-1\right\} $$ and the \textit{de Sitter 3-space} as 
$$S_1^3=\left\{ \bx\in\R_1^4 \:|\: \langle \bx,\bx\rangle=1\right\}.$$
The submanifold    $\Delta_5$ in \cite{ CI2009} is defined as   $   \Delta_5=\left\{(\bnu_1,\bnu_2)\in S_1^3\times S_1^3\; |\;\langle \bnu_1,\bnu_2 \rangle =0 \right\} .$  

It is well known that there are two connected branches in   hyperbolic 3-space.  Here we only consider  curves or surfaces on  the  branch $ \left\{ \bx=(x_1,x_2,x_3,x_4)\in H^3 \:|\:   x_1>0\right\}  $  and remain denote it by $H^3$.

In this paper, we denote $U$   a simply connected open subset in $\R^2$ and
$ I $   an interval in $\R$.
\subsection{Hyperbolic framed surfaces}
\begin{definition}[\cite{ZP}] 
{\rm 
We define  $(\bx,\bn,\bs): U \rightarrow H^3 \times \Delta_5$ as a  {\it hyperbolic  framed surface}  if  $ \langle \bx_u,\bn\rangle(u,v) =\langle\bx_v ,\bn  \rangle(u,v)= \langle\bx  ,\bn  \rangle(u,v)=\langle\bx ,\bs \rangle(u,v)=0$  for all $(u,v)\in U$, where $\bx_u(u,v)=(\partial\bx/\partial u)(u,v)$, $\bx_v(u,v)=(\partial\bx/\partial v)(u,v)$.   
  $\bx :  U \rightarrow H^3$ is  a $\mathit{hyperbolic}$ $\mathit{framed}$ $\mathit{base}$ $\mathit{surface}$ if there exists   $(\bn,\bs) : U \rightarrow \Delta_5$ such that  $(\bx,\bn,\bs): U \rightarrow H^3 \times \Delta_5$  is a hyperbolic framed surface.
}
\end{definition} 
Let $\bt(u,v) = (\bx \wedge \bn \wedge\bs)(u,v) .$ Then    $\left\{ \bx , \bn , \bs , \bt \right\} $ is a moving frame along $\bx $. Thus, we get 
\begin{equation}\notag
 \begin{aligned}
\begin{pmatrix}
\bx_u(u,v)\\
\bn_u(u,v)\\
\bs_u(u,v)\\
\bt_u(u,v)
\end{pmatrix}
&=
\begin{pmatrix}
0&0&a_1(u,v)&b_1(u,v)\\
0&0&e_1(u,v)&f_1(u,v)\\
a_1(u,v)&-e_1(u,v)&0&g_1(u,v)\\
b_1(u,v)&-f_1(u,v)&-g_1(u,v)&0
\end{pmatrix}
\begin{pmatrix}
\bx(u,v)\\
\bn(u,v)\\
\bs(u,v)\\
\bt(u,v)
\end{pmatrix},
 \\
 \begin{pmatrix}
\bx_v(u,v)\\
\bn_v(u,v)\\
\bs_v(u,v)\\
\bt_v(u,v)
\end{pmatrix}
&=
\begin{pmatrix}
0&0&a_2(u,v)&b_2(u,v)\\
0&0&e_2(u,v)&f_2(u,v)\\
a_2(u,v)&-e_2(u,v)&0&g_2(u,v)\\
b_2(u,v)&-f_2(u,v)&-g_2(u,v)&0
\end{pmatrix}
\begin{pmatrix}
\bx(u,v)\\
\bn(u,v)\\
\bs(u,v)\\
\bt(u,v)
\end{pmatrix},
  \end{aligned}
\end{equation}
where  $a_i, b_i, e_i, f_i, g_i: U \rightarrow \R$, $i=1, 2$  are  smooth functions. They are called the \textit{basic invariants} of $(\bx, \bm {n}, \bs).$ 
\subsection{One-parameter families of hyperbolic   framed curves}
\begin{definition} 
{\rm We call  $(\bgamma ,\bnu_1,\bnu_2) : U\rightarrow H^3 \times \Delta_5$  a {\it one-parameter family of hyperbolic  framed curves with  respect to $ u $} (respectively, {\it with respect to $ v $}) if $\langle \bgamma ,\bnu_i\rangle(u,v)=\langle \bgamma_u ,\bnu_i \rangle(u,v) =0$ (respectively, $\langle \bgamma ,\bnu_i\rangle(u,v)=\langle \bgamma_v ,\bnu_i \rangle(u,v) =0$) for all $(u,v)\in U$, $i=1,2$, where $\bgamma_u(u,v)=(\partial\bgamma/\partial u)(u,v)$, $\bgamma_v(u,v)=(\partial\bgamma/\partial v)(u,v)$.  $\bgamma  : U \rightarrow H^3 $ is called a {\it one-parameter family of hyperbolic  framed base curves with  respect to $ u $} (respectively, {\it with respect to $ v $}) if there exists  $(\bnu_1,\bnu_2) : U\rightarrow \Delta_5$ such that $(\bgamma,\bnu_1,\bnu_2) : U \rightarrow H^3 \times \Delta_5$  is a one-parameter family of hyperbolic  framed curves with  respect to $ u $ (respectively,   with respect to $ v $).
}
\end{definition}
Define $\bmu(u,v) =(\bgamma\wedge\bnu_1  \wedge\bnu_2)(u,v) $ ,    then $\left\{\bgamma  ,\bnu_1 ,\bnu_2,\bmu \right\}$ is a moving frame along $\bgamma $.
 
Let $(\bgamma ,\bnu_1,\bnu_2)$ be a one-parameter family of hyperbolic framed curves with respect to $u$. We have the following formulas
\begin{flalign} 
\begin{pmatrix}
\bgamma_u(u,v)  \\
\bnu_{1u}(u,v)  \\
\bnu_{2u} (u,v)  \\
\bmu_u(u,v)  
\end{pmatrix}
&=
\begin{pmatrix}
0 & 0 & 0 & m_1(u,v) \\
0 & 0 & n_1(u,v)  & a_1(u,v)  \\
0 & -n_1(u,v)  & 0 & b_1(u,v)  \\
m_1(u,v)  & -a_1(u,v)  & -b_1(u,v)  & 0 \\
\end{pmatrix}
\begin{pmatrix}
\bgamma(u,v)   \\
\bnu_1(u,v)  \\
\bnu_2(u,v)  \\
\bmu(u,v) 
\end{pmatrix},\notag\\ 
\begin{pmatrix}
\bgamma_v(u,v) \\
\bnu_{1v}(u,v)   \\
\bnu_{2v}(u,v)   \\
\bmu_v(u,v) 
\end{pmatrix}
&=
\begin{pmatrix}
0 & P_1(u,v) & Q_1(u,v) &M_1(u,v)\\
P_1(u,v) & 0 & N_1(u,v)  & A_1(u,v)  \\
Q_1(u,v) & -N_1(u,v) & 0 & B_1(u,v)  \\
M_1(u,v)  & -A_1(u,v) & -B_1(u,v)  & 0 \\
\end{pmatrix}
\begin{pmatrix}
\bgamma(u,v)  \\
\bnu_1(u,v)  \\
\bnu_2(u,v)  \\
\bmu (u,v)
\end{pmatrix},\notag
\end{flalign}
where  $m_1, n_1, a_1, b_1, M_1,N_1,A_1,B_1,P_1,Q_1 : U \rightarrow \R$  are  smooth functions. They are called the  {\it  curvature} of the one-parameter family of hyperbolic framed
curves with respect to $ u $.

If $(\bgamma ,\bnu_1,\bnu_2)$ is a one-parameter family of hyperbolic framed curves with respect to $v$, we have  
\begin{flalign} 
	\begin{pmatrix}
		\bgamma_u(u,v)  \\
		\bnu_{1u}(u,v)  \\
		\bnu_{2u} (u,v)  \\
		\bmu_u(u,v)  
	\end{pmatrix}
	&=
	\begin{pmatrix}
		0 & P_2(u,v) & Q_2(u,v) & M_2(u,v) \\
		P_2(u,v) & 0 & N_2(u,v)  & A_2(u,v)  \\
		Q_2(u,v) & -N_2(u,v)  & 0 & B_2(u,v)  \\
		M_2(u,v)  & -A_2(u,v)  & -B_2(u,v)  & 0 \\
	\end{pmatrix}
	\begin{pmatrix}
		\bgamma(u,v)   \\
		\bnu_1(u,v)  \\
		\bnu_2(u,v)  \\
		\bmu(u,v) 
	\end{pmatrix},\notag\\ 
	\begin{pmatrix}
		\bgamma_v(u,v) \\
		\bnu_{1v}(u,v)   \\
		\bnu_{2v}(u,v)   \\
		\bmu_v(u,v) 
	\end{pmatrix}
	&=
	\begin{pmatrix}
		0 & 0 & 0 &m_2(u,v)\\
		0 & 0 & n_2(u,v)  & a_2(u,v)  \\
		0 & -n_2(u,v) & 0 & b_2(u,v)  \\
		m_2(u,v)  & -a_2(u,v) & -b_2(u,v)  & 0 \\
	\end{pmatrix}
	\begin{pmatrix}
		\bgamma(u,v)  \\
		\bnu_1(u,v)  \\
		\bnu_2(u,v)  \\
		\bmu (u,v)
	\end{pmatrix},\notag
\end{flalign}
where  smooth functions  $m_2, n_2, a_2, b_2, M_2,N_2,A_2,B_2,P_2,Q_2 : U \rightarrow \R$  are   called the  {\it  curvature} of the one-parameter family of hyperbolic framed
curves with respect to $ v $.
\subsection{Generalized framed surfaces in Euclidean 3-space}
Let $\R^3$ denote   Euclidean  3-space with the inner product $\bx\cdot\by=x_1y_1+x_2y_2+x_3y_3$   and the vector product $$\bx\times\by=\det\begin{pmatrix}  \be_1&\be_2&\be_3\\
	x_1&x_2&x_3\\
	y_1&y_2&y_3
\end{pmatrix},$$  where $\bx=(x_1,x_2,x_3) ,  \by=(y_1,y_2,y_3)\in \R^3$ and $
\left\{\be_1,\be_2,\be_3\right\}$ is the canonical basis of $\R^3$.  Denote  $\Delta= 
\left\{(\bx,\by)\in S^2\times S^2\; | \; \bx\cdot\by   =0 \right\} $, where $S^2=\left\{ \bx\in\R^3 \:|\: \bx\cdot\bx =1\right\} $.
\begin{definition}[\cite{TY2024}]
\rm
We call  $(\overline{\bx},\overline{\bnu}_1,\overline{\bnu}_2 ): U \rightarrow \R^3 \times \Delta$  a   {\it generalized framed surface} if  there exist smooth functions $\overline{\alpha},\overline{\beta}:U\rightarrow\R$ such that $ (\overline{\bx}_u\times\overline{\bx}_v)(u,v)=\overline{\alpha}(u,v)\overline{\bnu}_1(u,v)+\overline{\beta} (u,v)\overline{\bnu}_2(u,v)$ for all $(u,v)\in U$. $\overline{\bx}:U\rightarrow \R^3$ 
is a  {\it generalized framed base surface}  if there exists $(\overline{\bnu}_1,\overline{\bnu}_2):U\rightarrow\Delta$ such that $(\overline{\bx},\overline{\bnu}_1,\overline{\bnu}_2)$ is a  generalized framed surface.
\end{definition}
Let $\overline{\bnu}_3(u,v)=(\overline{\bnu}_1 \times\overline{\bnu}_2)(u,v)$,  then $\left\{  \overline{\bnu}_1, \overline{\bnu}_2, \overline{\bnu}_3\right\} $ is a moving frame along $\overline{\bx} $.

\subsection{Lightcone framed surfaces in Lorentz-Minkowski 3-space}
The submanifold $\Delta_4$   in \cite{ CI2009} is defined as $  \Delta_4=\left\{(\bnu_1,\bnu_2)\in LC^{*}\times LC^{*}\; |\;\langle \bnu_1,\bnu_2 \rangle =-2 \right\}  $, where $LC^{*}=\left\{ \bx\in\R_1^3\backslash \left\{\bm{0}\right\} \:|\: \langle \bx,\bx\rangle=0\right\}$. 
\begin{definition}[\cite{LPT2025}]
\rm
We call  $(\widetilde{\bx}, \bm{\ell}^+,\bm{\ell}^- ): U \rightarrow \R_1^3 \times \Delta_4$  a   {\it lightcone framed surface} if  there exist smooth functions $\widetilde{\alpha},\widetilde{\beta}:U\rightarrow\R$ such that $ (\widetilde{\bx}_u\wedge\widetilde{\bx}_v)(u,v)=\widetilde{\alpha}(u,v)\bm{\ell}^+(u,v)+\widetilde{\beta} (u,v)\bm{\ell}^-(u,v)$ for all $(u,v)\in U$. $\widetilde{\bx}:U\rightarrow \R_1^3$ 
is a  {\it lightcone framed base surface}  if there exists $(\bm{\ell}^+,\bm{\ell}^-):U\rightarrow\Delta_4$ such that $(\widetilde{\bx}, \bm{\ell}^+,\bm{\ell}^- )$ is a  lightcone framed surface.
\end{definition} 
\section{Hyperbolic generalized framed surfaces}\label{S3}
We give a definition of a hyperbolic generalized framed surface and investigate the properties of it.
\begin{definition}
\rm We say that $(\bx,\bnu_1,\bnu_2): U \rightarrow H^3 \times \Delta_5$ is a  {\it hyperbolic generalized framed surface}  if $\langle \bx ,\bnu_1\rangle(u,v)=\langle \bx ,\bnu_2\rangle(u,v)=0$ and there exist smooth functions $\alpha,\beta:U\rightarrow\R$  such that $(\bx\wedge\bx_u\wedge\bx_v)(u,v)=\alpha(u,v)\bnu_1(u,v)+\beta(u,v)\bnu_2(u,v)$ for all $(u,v)\in U$.     
We define $\bx :  U \rightarrow H^3$ as  a {\it hyperbolic generalized framed base surface}  if there exists   $(\bnu_1,\bnu_2) : U \rightarrow \Delta_5$ such that  $(\bx,\bnu_1,\bnu_2): U \rightarrow H^3 \times \Delta_5$  is a hyperbolic generalized framed surface.
\end{definition}
Let $\bnu_3(u,v) = (\bx \wedge \bnu_1 \wedge\bnu_2)(u,v) .$ Then    $\left\{ \bx , \bnu_1 , \bnu_2 , \bnu_3 \right\} $ is a moving frame along $\bx $. Thus, we get 
\begin{equation}\label{eq1} 
{
\begin{pmatrix}
\bx_u(u,v)\\
\bnu_{1u}(u,v)\\
\bnu_{2u}(u,v)\\
\bnu_{3u}(u,v)
\end{pmatrix}
=
\begin{pmatrix}
0&a_1(u,v)&b_1(u,v)&c_1(u,v)\\
a_1(u,v)&0&e_1(u,v)&f_1(u,v)\\
b_1(u,v)&-e_1(u,v)&0&g_1(u,v)\\
c_1(u,v)&-f_1(u,v)&-g_1(u,v)&0
\end{pmatrix}
\begin{pmatrix}
\bx(u,v)\\
\bnu_1(u,v)\\
\bnu_2(u,v)\\
\bnu_3(u,v)
\end{pmatrix},
}
\end{equation}

\begin{equation} \label{eq2} 
{\begin{pmatrix}
\bx_v(u,v)\\
\bnu_{1v}(u,v)\\
\bnu_{2v}(u,v)\\
\bnu_{3v}(u,v)
\end{pmatrix}
=
\begin{pmatrix}
0&a_2(u,v)&b_2(u,v)&c_2(u,v)\\
a_2(u,v)&0&e_2(u,v)&f_2(u,v)\\
b_2(u,v)&-e_2(u,v)&0&g_2(u,v)\\
c_2(u,v)&-f_2(u,v)&-g_2(u,v)&0
\end{pmatrix}
\begin{pmatrix}
\bx(u,v)\\
\bnu_1(u,v)\\
\bnu_2(u,v)\\
\bnu_3(u,v)
\end{pmatrix},
}
\end{equation}
where  $a_i, b_i,c_i, e_i, f_i, g_i: U \rightarrow \R$, $i=1, 2$  are  smooth functions  with $a_1b_2-a_2b_1=0$. They are called the {\it basic invariants} of $(\bx, \bm {\nu}_1, \bnu_2).$ Denote $$\ba(u,v)=\begin{pmatrix}a_1(u,v)~a_2(u,v) \end{pmatrix}^{T}, \bb(u,v)=\begin{pmatrix}b_1(u,v)~b_2(u,v) \end{pmatrix}^{T}, \bc(u,v)=\begin{pmatrix}c_1(u,v)~c_2(u,v) \end{pmatrix}^{T},$$ where $\ba^T$ denotes  the transpose of $\ba$. Thus we can get
\begin{equation} \notag
\alpha(u,v)
=
\det\begin{pmatrix}
\bb(u,v)~\bc(u,v)  
\end{pmatrix}
,\;
\beta(u,v)
=
\det\begin{pmatrix}
\bc (u,v)~\ba (u,v) \end{pmatrix}
.
\end{equation}
Since $(\bx,\bnu_1,\bnu_2)$ is smooth,  $\bx_{uv}=\bx_{vu}$, $\bnu_{1uv}=\bnu_{1vu}$ and $\bnu_{2uv}=\bnu_{2vu}$ hold. It follows that we have the  integrability  conditions:
\begin{equation}\label{eq3}
\left\{	\begin{array}{lr}
a_{1v}-b_1e_2-c_1f_2=a_{2u}-b_2e_1-c_2f_1,\\
b_{1v}+a_1e_2-c_1g_2=b_{2u}+a_2e_1-c_2g_1,\\
c_{1v}+a_1f_2+b_1g_2=c_{2u}+a_2f_1+b_2g_1,
\end{array} 
\right.
\end{equation}
\begin{equation}\label{eq4}
\left\{	\begin{array}{lr}
e_{1v}-f_1g_2=e_{2u}-f_2g_1,\\
f_{1v}+e_1g_2+a_1c_2=f_{2u}+e_2g_1+a_2c_1,\\
g_{1v}-e_1f_2+b_1c_2=g_{2u}-e_2f_1+b_2c_1.
\end{array} 
\right.
\end{equation}

Then we give  fundamental theorems for hyperbolic generalized framed surfaces using their basic invariants.
\begin{definition}\rm
Let $(\bx, \bnu_1, \bnu_2),(\overline{\bx}, \overline{\bnu}_1, \overline{\bnu}_2):U\rightarrow H^3 \times \Delta_5$ be two hyperbolic generalized framed surfaces. We say that  $(\bx, \bnu_1, \bnu_2)$ and $(\overline{\bx}, \overline{\bnu}_1, \overline{\bnu}_2)$ are {\it{congruent}} as hyperbolic generalized framed surfaces  if there exists     $\bm{A}\in SO(1,3)$   such that 
$\overline{\bx}(u,v) = \bm{A}(\bx(u,v))$,  $\overline{\bnu}_1(u,v) = \bm{A}(\bnu_1(u,v))$,   $\overline{\bnu}_2(u,v) = \bm{A}(\bnu_2(u,v))$
for all $(u,v)\in U$. Here,  $$SO(1,3)=\left\{\bm{A}\in \bm{M}_4(\R)~\arrowvert~\bm{A}^T\bm{GA} = \bm{G}, \;\det(\bm{A}) = 1,\; \bm{G} =
\begin{pmatrix}
-1 & 0 & 0 & 0 \\
0 & 1 & 0 & 0 \\
0 & 0 & 1 & 0 \\
0 & 0 & 0 & 1
\end{pmatrix}\right\}.$$ 
\end{definition}

\begin{theorem}[Existence theorem for hyperbolic generalized framed surfaces]
Let   $a_i, b_i,c_i, e_i$, $f_i, g_i: U\rightarrow \R $, $i=1,2$ be smooth functions with $a_1b_2-a_2b_1=0$ and the integrability conditions $(\ref{eq3})$ and $(\ref{eq4})$. Then there exists a hyperbolic generalized framed surface $(\bx,\bnu_1,\bnu_2): U \rightarrow H^3 \times \Delta_5$ whose associated basic invariants are  $a_i, b_i,c_i, e_i, f_i, g_i: U\rightarrow \R $, $i=1,2$.
\begin{proof}
Since the integrability conditions  (\ref{eq3}) and (\ref{eq4}) exist,  the systems (\ref{eq1}) and (\ref{eq2}) have a solution  $(\bx, \bnu_1, \bnu_2,\bnu_3):U\rightarrow H^3 \times \Delta_5\times S_1^3$, which satisfies $\bx, \bnu_1, \bnu_2,\bnu_3$ are pseudo-orthogonal and $\bx\wedge\bnu_1\wedge \bnu_2=\bnu_3$. It follows that $(\bx\wedge\bx_u\wedge\bx_v)(u,v)=(b_1c_2-b_2c_1)(u,v)\bnu_1(u,v)+(c_1a_2-c_2a_1)(u,v)\bnu_2(u,v)$. 
Therefore, $(\bx, \bnu_1, \bnu_2):U\rightarrow H^3 \times \Delta_5$ is a hyperbolic generalized framed surface whose basic invariants are   $a_i, b_i,c_i, e_i, f_i, g_i $, $i=1,2$.
\end{proof}

\end{theorem}

\begin{theorem}[Uniqueness theorem for hyperbolic generalized framed surfaces]
Let $(\bx, \bnu_1, \bnu_2)$, $(\overline{\bx}, \overline{\bnu}_1, \overline{\bnu}_2):U\rightarrow H^3 \times \Delta_5$ be two hyperbolic generalized framed surfaces with the same basic invariants  $a_i, b_i,c_i, e_i, f_i, g_i: U\rightarrow \R $, $i=1,2$. Then $(\bx, \bnu_1, \bnu_2)$ and $(\overline{\bx}, \overline{\bnu}_1, \overline{\bnu}_2)$ are congruent as hyperbolic generalized framed surfaces.
\begin{proof}
Let $(\bx, \bnu_1, \bnu_2)$ and $(\overline{\bx}, \overline{\bnu}_1, \overline{\bnu}_2)$ be two  hyperbolic generalized framed surfaces  with the same basic invariants $a_i, b_i,c_i, e_i, f_i, g_i,$ $ i = 1,2.$ Let $\bnu_3(u,v)=(\bx \wedge \bnu_1 \wedge\bnu_2)(u,v)$ and  $\overline{\bnu}_3(u,v)=(\overline{\bx} \wedge \overline{\bnu}_1 \wedge\overline{\bnu}_2)(u,v).$ For some fixed point $(u_0, v_0) \in U,$ there exists $\bm{A}\in SO(1,3)$   satisfying
$\overline{\bx}(u_0,v_0) = \bm{A}\left(\bx(u_0,v_0)\right)$,   $\overline{\bnu}_1(u_0,v_0) = \bm{A}(\bnu_1(u_0,v_0))$, $ \overline{\bnu}_2(u_0,v_0) = \bm{A}(\bnu_2(u_0,v_0))$.
Then  we have $\overline{\bnu}_3(u_0,v_0) = \bm{A}(\bnu_3(u_0,v_0)).$
Consider two maps $(\bm{A}(\bx),\bm{A}(\bnu_1), \bm{A}(\bnu_2),\bm{A}(\bnu_3)) : U \rightarrow {H}^3 \times \Delta_5 \times S_1^3$ and $(\overline{\bx}, \overline{\bnu}_1, \overline{\bnu}_2,\overline{\bnu}_3) : U \rightarrow {H}^3 \times \Delta_5 \times S_1^3.$ They are both solutions of the systems (\ref{eq1}) and (\ref{eq2}). Thus  $(\bm{A}(\bx),\bm{A}(\bnu_1), \bm{A}(\bnu_2),\bm{A}(\bnu_3)) = (\overline{\bx}, \overline{\bnu}_1, \overline{\bnu}_2,\overline{\bnu}_3).$ That is to say $(\bx, \bnu_1, \bnu_2)$ and $(\overline{\bx}, \overline{\bnu}_1, \overline{\bnu}_2)$ are congruent as hyperbolic generalized framed surfaces.
\end{proof}
\end{theorem}
\subsection{Relations between hyperbolic generalized framed surfaces, hyperbolic framed surfaces and one-parameter families of hyperbolic framed curves.}
By the relations between basic invariants of hyperbolic generalized framed surfaces and hyperbolic framed surfaces, we have the following conclusion.
\begin{proposition}\label{prop5}
\begin{itemize}
\item[\rm (1)]  Let   $(\bx,\bn,\bs): U \rightarrow H^3 \times \Delta_5$  be  a   hyperbolic  framed surface with basic invariants $a_i,b_i,e_i,f_i,g_i$, $i=1,2$. Then $(\bx,\bn,\bs)$ is a  hyperbolic generalized framed surface with $\alpha(u,v)=(a_1b_2-a_2b_1)(u,v)$, $\beta(u,v)=0$.
\item[\rm(2)]  Let   $(\bx,\bnu_1,\bnu_2): U \rightarrow H^3 \times \Delta_5$  be  a   hyperbolic generalized  framed surface with basic invariants $a_i,b_i,c_i,e_i,f_i,g_i$, $i=1,2$.

 {\rm(i)} If $a_1(u,v)=a_2(u,v)=0$ for all $(u,v)\in U$, then $(\bx,\bnu_1,\bnu_2)$  is  a  hyperbolic framed surface.

 {\rm (ii)} If $b_1(u,v)=b_2(u,v)=0$ for all $(u,v)\in U$, then $(\bx,\bnu_2,\bnu_1)$  is  a  hyperbolic framed surface.
\end{itemize}
\begin{proof}
 {\rm (1)} If $(\bx,\bn,\bs)$   is  a   hyperbolic  framed surface, the normal vector of $\bx$ is $(\bx\wedge\bx_u\wedge\bx_v)(u,v)=(a_1b_2-a_2b_1)(u,v)\bn(u,v)$, which means that $(\bx,\bn,\bs)$ is a hyperbolic generalized framed surface with $\alpha(u,v)=(a_1b_2-a_2b_1)(u,v)$, $\beta(u,v)=0$. 
 
 {\rm (2)} If the basic invariants of hyperbolic generalized framed surface $(\bx,\bnu_1,\bnu_2)$ satisfy $a_1(u,v)=a_2(u,v)=0$ for all $(u,v)\in U$,   we can get $ \langle \bx_u,\bnu_1\rangle (u,v)= \langle \bx_v,\bnu_1\rangle (u,v)=\langle \bx,\bnu_1\rangle (u,v)=\langle \bx,\bnu_2\rangle (u,v)=\langle \bnu_1,\bnu_2\rangle (u,v)=0$ for all $(u,v)\in U$. It follows that $(\bx,\bnu_1,\bnu_2)$ is a hyperbolic framed surface. If $b_1(u,v)=b_2(u,v)=0$ for all $(u,v)\in U$, similar discussion can lead to the conclusion.
 \end{proof}
\end{proposition}
\begin{remark}\rm
Regular surfaces in hyperbolic 3-space  are hyperbolic framed base surfaces, and therefore, regular surfaces are also hyperbolic generalized framed base surfaces.
\end{remark}
We also discuss the relations between hyperbolic generalized framed surfaces and one-parameter families of hyperbolic framed curves.
\begin{proposition}\label{prop6}
\begin{itemize}
\item[\rm(1)]{\rm(i)} Let $(\bgamma ,\bnu_1,\bnu_2) : U\rightarrow H^3 \times \Delta_5$  be a one-parameter family of hyperbolic  framed curves with  respect to $ u $. Then $(\bgamma ,\bnu_1,\bnu_2) $ is a  hyperbolic generalized framed surface with $\alpha(u,v)=-m_1(u,v)Q_1(u,v)$, $\beta(u,v)=m_1(u,v)P_1(u,v)$.

 {\rm(ii)} Let $(\bgamma ,\bnu_1,\bnu_2) : U\rightarrow H^3 \times \Delta_5$  be a one-parameter family of hyperbolic  framed curves with  respect to $ v $. Then $(\bgamma ,\bnu_1,\bnu_2) $ is a  hyperbolic generalized framed surface with $\alpha(u,v)=m_2(u,v)Q_2(u,v)$, $\beta(u,v)=-m_2(u,v)P_2(u,v)$.
\item[\rm(2)]  Let   $(\bx,\bnu_1,\bnu_2): U \rightarrow H^3 \times \Delta_5$  be  a   hyperbolic generalized  framed surface with basic invariants $a_i,b_i,c_i,e_i,f_i,g_i$, $i=1,2$.

{\rm(i)} If $a_1(u,v)=b_1(u,v)=0$ for all $(u,v)\in U$, then $(\bx,\bnu_1,\bnu_2) $ and $(\bx,\bnu_2,\bnu_1)$ are  one-parameter families of hyperbolic  framed curves with  respect to $ u $.

{\rm(ii)} If $a_2(u,v)=b_2(u,v)=0$ for all $(u,v)\in U$, then $(\bx,\bnu_1,\bnu_2) $ and $(\bx,\bnu_2,\bnu_1)$ are  one-parameter families of hyperbolic  framed curves with  respect to $ v $.
\end{itemize}
\begin{proof}
 {\rm (1)} (i) If $(\bgamma ,\bnu_1,\bnu_2)$   is  a one-parameter family of hyperbolic  framed curves with  respect to $ u $,  then $(\bgamma\wedge\bgamma_u\wedge\bgamma_v)(u,v)=-m_1(u,v)Q_1(u,v)\bnu_1(u,v)+m_1(u,v)P_1(u,v)\bnu_2(u,v)$, and consequently,   $(\bgamma ,\bnu_1,\bnu_2)$ is a hyperbolic generalized framed surface with $\alpha(u,v)=-m_1(u,v)Q_1(u,v)$, $\beta(u,v)=m_1(u,v)P_1(u,v)$.
 For case (ii), a similar analysis can yield to the   conclusion.
 
 {\rm (2)}  If the basic invariants of hyperbolic generalized framed surface $(\bx,\bnu_1,\bnu_2)$ satisfy $a_1(u,v)=b_1(u,v)=0$ for all $(u,v)\in U$,   we can get $ \langle \bx_u,\bnu_1\rangle (u,v)= \langle \bx_u,\bnu_2\rangle (u,v)=\langle \bx,\bnu_1\rangle (u,v)=\langle \bx,\bnu_2\rangle (u,v)=\langle \bnu_1,\bnu_2\rangle (u,v)=0$ for all $(u,v)\in U$. It follows that $(\bx,\bnu_1,\bnu_2) $ and $(\bx,\bnu_2,\bnu_1)$ are  one-parameter families of hyperbolic  framed curves with  respect to $ u $. If $a_2(u,v)=b_2(u,v)=0$ for all $(u,v)\in U$, similar discussion can lead to the conclusion.
\end{proof}
\end{proposition}
 
\subsection{Reflection and rotation of frame  as well as a parameter change in  $U$.}

By direct calculations, we have the following.
\begin{proposition}
Let   $(\bx,\bnu_1,\bnu_2): U \rightarrow H^3 \times \Delta_5$  be  a   hyperbolic generalized  framed surface with basic invariants $a_i,b_i,c_i,e_i,f_i,g_i$, $i=1,2$. Then we have
\begin{itemize}
\item [\rm(1)] $(\bx,-\bnu_1,\bnu_2): U \rightarrow H^3 \times \Delta_5$  is  a   hyperbolic generalized  framed surface with basic invariants $-a_i,b_i,-c_i,-e_i,f_i,-g_i$, $i=1,2$.
\item [\rm(2)] $(\bx,-\bnu_1,-\bnu_2): U \rightarrow H^3 \times \Delta_5$  is  a   hyperbolic generalized  framed surface with basic invariants $-a_i,-b_i, c_i,e_i,-f_i,-g_i$, $i=1,2$.
\item [\rm(3)] $(\bx,\bnu_2,\bnu_1): U \rightarrow H^3 \times \Delta_5$  is  a   hyperbolic generalized  framed surface with basic invariants $b_i,a_i, -c_i,-e_i,-g_i,-f_i$, $i=1,2$.
\end{itemize}
\end{proposition}
Let   $(\bx,\bnu_1,\bnu_2): U \rightarrow H^3 \times \Delta_5$  be  a   hyperbolic generalized  framed surface with basic invariants $a_i,b_i,c_i,e_i,f_i,g_i$, $i=1,2$. We denote
\begin{equation}\label{eqr}
\begin{pmatrix}
\bnu_1^{\theta}(u,v)\\
\bnu_2^{\theta}(u,v)
\end{pmatrix}
=\begin{pmatrix}
\cos\theta(u,v)&-\sin\theta(u,v)\\
\sin\theta(u,v)&\cos\theta(u,v)
\end{pmatrix}
\begin{pmatrix}
\bnu_1 (u,v)\\
\bnu_2 (u,v)
\end{pmatrix},
\end{equation}
where $\theta:U\rightarrow\R$ is a smooth function. It can be verified that
$\langle \bx, \bnu_1^{\theta}\rangle(u,v)=\langle \bx, \bnu_2^{\theta}\rangle(u,v)=0$ and
\begin{equation}\notag
(\bx\wedge\bx_u\wedge\bx_v)(u,v)
=(\alpha\cos\theta-\beta\sin\theta)(u,v)\bnu_1^{\theta}(u,v)+(\alpha\sin\theta+\beta\cos\theta)(u,v)\bnu_2^{\theta}(u,v),
\end{equation}
then $(\bx,\bnu_1^{\theta},\bnu_2^{\theta}): U \rightarrow H^3 \times \Delta_5$  is also  a   hyperbolic generalized  framed surface with 
$\alpha^{\theta}(u,v)=(\alpha\cos\theta-\beta\sin\theta)(u,v)$ and $\beta^{\theta}(u,v)=(\alpha\sin\theta+\beta\cos\theta)(u,v)$.
\begin{proposition}\label{3.10}
Under the above notations,  for any smooth function $\theta:U\rightarrow\R$, the basic invariants  $ (a^{\theta}_i,b^{\theta}_i,c^{\theta}_i,e^{\theta}_i,f^{\theta}_i,g^{\theta}_i) $  $ (i=1,2) $ of $(\bx,\bnu^{\theta}_1,\bnu^{\theta}_2)$  are given by
\begin{equation}\notag
	\begin{aligned}
\begin{pmatrix}
	a_1^{\theta}(u,v)&b_1^{\theta}(u,v)&c_1^{\theta}(u,v)\\
	a_2^{\theta}(u,v)&b_2^{\theta}(u,v)&c_2^{\theta}(u,v)
\end{pmatrix}
&=	\begin{pmatrix}
	a_1(u,v)&b_1(u,v)&c_1(u,v)\\
	a_2(u,v)&b_2(u,v)&c_2(u,v)
\end{pmatrix}
\begin{pmatrix}
	\cos\theta(u,v)&\sin\theta(u,v)&0\\
	-\sin\theta(u,v)&\cos\theta(u,v)&0\\
	0&0&1
\end{pmatrix},\\
\begin{pmatrix}
	e_1^{\theta}(u,v) \\
	e_2^{\theta}(u,v) 
\end{pmatrix}
&=	\begin{pmatrix}
	e_1(u,v)-\theta_u(u,v)\\
	e_2(u,v)-\theta_v(u,v)
\end{pmatrix},\\
\begin{pmatrix}
	f_1^{\theta}(u,v)&g_1^{\theta}(u,v)\\
	f_2^{\theta}(u,v)&g_2^{\theta}(u,v) 
\end{pmatrix}
&=
\begin{pmatrix}
	f_1(u,v)&g_1(u,v) \\
	f_2(u,v)&g_2 (u,v)
\end{pmatrix}
\begin{pmatrix}
	\cos\theta(u,v)&\sin\theta(u,v) \\
	-\sin\theta(u,v)&\cos\theta(u,v) 
\end{pmatrix}.
	\end{aligned}
\end{equation}
\end{proposition}
Next, we consider a parameter change of the domain $U$.
\begin{proposition}
Let   $(\bx,\bnu_1,\bnu_2): U \rightarrow H^3 \times \Delta_5$  be  a   hyperbolic generalized  framed surface with   basic invariants $a_i,b_i,c_i,e_i,f_i,g_i$, $i=1,2$ and $\bm{\phi}:V\rightarrow U$ be a paremeter change, where $\bm{\phi}(p,q)=(u(p,q),v(p,q))$. Then $(\overline{\bx},\overline{\bnu}_1,\overline{\bnu}_2)=(\bx,\bnu_1,\bnu_2)\circ\bm{\phi}:V\rightarrow H^3\times\Delta_5$  is a hyperbolic generalized  framed surface with
\begin{equation}\notag
\overline{\alpha}(p,q)=\alpha\circ\bm{\phi}(p,q) \det\begin{pmatrix}
	u_p(p,q)&v_p(p,q)\\
	u_q(p,q)&v_q(p,q)\end{pmatrix}
,\;
\overline{\beta}(p,q)=\beta\circ\bm{\phi}(p,q) \det\begin{pmatrix}
	u_p(p,q)&v_p(p,q)\\
	u_q(p,q)&v_q(p,q)\end{pmatrix}
\end{equation}
and the basic invariants
\begin{equation}\notag
\begin{pmatrix}
	\overline{a}_1(p,q)&\overline{b}_1(p,q) &\overline{c}_1(p,q) \\
	\overline{a}_2 (p,q)&\overline{b}_2(p,q)&\overline{c}_2(p,q) 
\end{pmatrix}=
\begin{pmatrix}
	u_p(p,q)&v_p(p,q)\\
	u_q(p,q)&v_q(p,q)
\end{pmatrix}
\begin{pmatrix}
	a_1\circ\bm{\phi}(p,q)&b_1\circ\bm{\phi}(p,q) &c_1\circ\bm{\phi}(p,q) \\
	a_2\circ\bm{\phi} (p,q)&b_2\circ\bm{\phi}(p,q)&c_2\circ\bm{\phi}(p,q) 
\end{pmatrix},
\end{equation}
\begin{equation}\notag
\begin{pmatrix}
	\overline{e}_1(p,q)&\overline{f}_1(p,q) &\overline{g}_1(p,q) \\
	\overline{e}_2 (p,q)&\overline{f}_2(p,q)&\overline{g}_2(p,q) 
\end{pmatrix}=
\begin{pmatrix}
	u_p(p,q)&v_p(p,q)\\
	u_q(p,q)&v_q(p,q)
\end{pmatrix}
\begin{pmatrix}
	e_1\circ\bm{\phi}(p,q)&f_1\circ\bm{\phi}(p,q) &g_1\circ\bm{\phi}(p,q) \\
	e_2 \circ\bm{\phi}(p,q)&f_2\circ\bm{\phi}(p,q)&g_2\circ\bm{\phi}(p,q) 
\end{pmatrix}.
\end{equation}
\end{proposition}
\subsection{Relations between hyperbolic generalized framed base surfaces, hyperbolic framed  base surfaces, one-parameter families of hyperbolic framed base curves  and smooth surfaces.}
Firstly, we  give   conditions for a hyperbolic generalized framed base surface to be a hyperbolic framed base surface or a one-parameter family of hyperbolic framed base curves.
\begin{proposition}\label{th11}
Let $(\bx,\bnu_1,\bnu_2):U\rightarrow H^3\times\Delta_5$ be a hyperbolic generalized framed surface with    basic invariants $a_i,b_i,c_i,e_i,f_i,g_i$, $i=1,2$. If there exists a smooth function $\theta:U\rightarrow \R$ such that $(a_1\cos\theta-b_1\sin\theta)(u,v)=(a_2\cos\theta-b_2\sin\theta)(u,v)=0$   holds for all $(u,v)\in U$, then $\bx$ is a hyperbolic framed base surface.
\begin{proof}
Consider the hyperbolic generalized framed surface $(\bx,\bnu_1^{\theta},\bnu_2^{\theta})$, where $\bnu_1^{\theta}$, $\bnu_2^{\theta}$ are defined as equation (\ref{eqr}). By Proposition \ref{3.10}, we obtain the basic invariants of $(\bx,\bnu_1^{\theta},\bnu_2^{\theta})$ satisfy $a_1^{\theta}(u,v)=a_2^{\theta}(u,v)=0$ for all $(u,v)\in U$.	 It follows from Proposition \ref{prop5} that $(\bx,\bnu_1^{\theta},\bnu_2^{\theta})$ is a hyperbolic framed surface and $\bx$ is a hyperbolic framed base surface.
\end{proof}
\end{proposition}
\begin{proposition}\label{prop13}
Let $(\bx,\bnu_1,\bnu_2):U\rightarrow H^3\times\Delta_5$ be a hyperbolic generalized framed surface with    basic invariants $a_i,b_i,c_i,e_i,f_i,g_i$, $i=1,2$. 
\begin{itemize}
\item[\rm(1)] Assume there exists a smooth function $\theta:U\rightarrow\R$ such that $(a_1\cos\theta-b_1\sin\theta)(u,v)=0$ for all $(u,v)\in U$.
 
{\rm(i)}  If $c_1(u,v)=0$ for all $(u,v)\in U$, then $\bx$ is a one-parameter family of hyperbolic framed base curves with respect to $u$.

{\rm(ii)}  If there exists a point $(u_0,v_0)\in U$ such that $ c_1(u_0,v_0)\neq0 $, then $\bx$ is a one-parameter family of hyperbolic framed base curves with respect to $u$ at least locally around $(u_0,v_0)$.
 
\item[\rm(2)] Assume there exists a smooth function $\theta:U\rightarrow\R$ such that $(a_2\cos\theta-b_2\sin\theta)(u,v)=0$ for all $(u,v)\in U$.
 
 {\rm(i)} If $c_2(u,v)=0$ for all $(u,v)\in U$, then $\bx$ is a one-parameter family of hyperbolic framed base curves with respect to $v$.

 {\rm(ii)}  If there exists a point $(u_0,v_0)\in U$ such that $ c_2(u_0,v_0)\neq0 $, then $\bx$ is a one-parameter family of hyperbolic framed base curves with respect to $v$ at least locally around $(u_0,v_0)$.
 
\end{itemize}
\begin{proof}
{\rm(1) (i)} If $(a_1\cos\theta-b_1\sin\theta)(u,v)=c_1(u,v)=0$ for all $(u,v)\in U$, we have $a_1^{\theta}(u,v)=c_1^{\theta}(u,v)=0$ for all $(u,v)\in U$ by Proposition \ref{3.10}, where $ a_1^{\theta}, c_1^{\theta} $ are the basic invariants of $(\bx,\bnu_1^{\theta},\bnu_2^{\theta})$. Denote $\bnu^{\theta}_3(u,v)=(\bx\wedge\bnu_1^{\theta} \wedge\bnu_2^{\theta})(u,v)$. Then Proposition \ref{prop6} shows that $(\bx,\bnu_1^{\theta},\bnu_3^{\theta})$ and $(\bx,\bnu_3^{\theta},\bnu_1^{\theta})$ are   one-parameter families of hyperbolic framed   curves with respect to $u$ and $\bx$ is a one-parameter family of hyperbolic framed base curves with respect to $u$.

{\rm(ii)} If there exists a point $(u_0,v_0)\in U$ such that $ c_1(u_0,v_0)\neq0 $, then $c_1(u,v)\neq0$ at least locally around $(u_0,v_0)$. Take $ \overline{\bnu}_1(u,v)=\bnu_1^{\theta}(u,v)$ and $$ \overline{\bnu}_2(u,v)=\frac{-b_1^{\theta}\bnu^{\theta}_3 +c_1\bnu_2^{\theta}}{\sqrt{(b_1^{\theta})^2+c_1^2}}(u,v).$$   We can verify that  $\langle \bx_u,\overline{\bnu}_1\rangle(u,v)=\langle \bx_u,\overline{\bnu}_2\rangle(u,v)=\langle \overline{\bnu}_1,\overline{\bnu}_2\rangle(u,v)=\langle \bx,\overline{\bnu}_1\rangle(u,v)=\langle \bx,\overline{\bnu}_2\rangle(u,v)=0$ at least locally around $(u_0,v_0)$, which implies that $(\bx,\overline{\bnu}_1,\overline{\bnu}_2)$ and $(\bx,\overline{\bnu}_2,\overline{\bnu}_1)$ are  one-parameter families of hyperbolic framed   curves with respect to $u$ and $\bx$ is a one-parameter family of hyperbolic framed base curves with respect to $u$ at least locally around $(u_0,v_0)$.	

For case (2), a similar analysis can lead to the   conclusion.
\end{proof}
\end{proposition}
We give a condition for a smooth surface to be a hyperbolic generalized framed base surface.
\begin{theorem}\label{th7}
Let $\bx:U\rightarrow H^3$ be a smooth surface. Then $\bx$ is a hyperbolic generalized framed base surface if and only if there exists a smooth map $\bnu:U\rightarrow S_1^3$ such that $\langle \bx,\bnu\rangle(u,v)=\langle \bx\wedge\bx_u\wedge\bx_v,\bnu\rangle(u,v)=0$ for all $(u,v)\in U$.
\begin{proof}
Since  $\bx$ is a hyperbolic generalized framed base surface, there exists $(\bnu_1,\bnu_2) : U \rightarrow \Delta_5$ such that $(\bx,\bnu_1,\bnu_2)$ is a hyperbolic generalized framed surface. Then there are  smooth functions $\alpha,\beta:U\rightarrow\R$ such that $(\bx\wedge\bx_u\wedge\bx_v)(u,v)=\alpha(u,v)\bnu_1(u,v)+\beta(u,v)\bnu_2(u,v)$. Let $\bnu(u,v)=(\bx\wedge\bnu_1\wedge\bnu_2)(u,v)$. Then $\langle \bx,\bnu\rangle(u,v)=\langle \bx\wedge\bx_u\wedge\bx_v,\bnu\rangle(u,v)=0$ for all $(u,v)\in U$. 

Conversely, assume that $\bnu:U\rightarrow S_1^3$ satisfying $\langle \bx,\bnu\rangle(u,v)=\langle \bx\wedge\bx_u\wedge\bx_v,\bnu\rangle(u,v)=0$ for all $(u,v)\in U$. Denote $\bnu(u,v)=(a,b,c,d)(u,v)$ and $\bx(u,v)=(x_0,x_1,x_2,x_3)(u,v)$, then there exist smooth functions $\theta,\varphi:U\rightarrow\R$ such that $\bnu$ can be expressed   as
\begin{equation}\notag
\bnu(u,v)=\left(a,\sqrt{1+a^2}\sin\theta\sin\varphi,\sqrt{1+a^2}\sin\theta\cos\varphi,\sqrt{1+a^2}\cos\theta\right)(u,v).  
\end{equation}
Since  \begin{equation}\notag
\langle \bx,\bnu\rangle(u,v)=\left(-ax_0+x_1\sqrt{1+a^2}\sin\theta\sin\varphi+x_2\sqrt{1+a^2}\sin\theta\cos\varphi+x_3\sqrt{1+a^2}\cos\theta\right)(u,v)=0,
\end{equation}
we have  \begin{equation} \label{eq5}
\left(x_1\sin\theta\sin\varphi+x_2\sin\theta\cos\varphi+x_3\cos\theta\right)(u,v)=\frac{ax_0}{\sqrt{1+a^2}}(u,v).
\end{equation}
Denote \begin{equation}\notag
\begin{aligned}
&\overline{\bnu}_1(u,v)=(0,\cos\theta\sin\varphi,\cos\theta\cos\varphi,-\sin\theta)(u,v),\\
&\overline{\bnu}_2(u,v)=(0,\cos\varphi,-\sin\varphi,0)(u,v).
\end{aligned}
\end{equation}
By equation (\ref{eq5}), we have 
\begin{equation}\notag
\begin{aligned}
&(\bx\wedge\overline{\bnu}_1\wedge\overline{\bnu}_2)(u,v)\\
=&\left(x_1\sin\theta\sin\varphi+x_2\sin\theta\cos\varphi+x_3\cos\theta,x_0\sin\theta\sin\varphi,x_0\sin\theta\cos\varphi,x_0\cos\theta\right)(u,v)\\
=&\left(\frac{ax_0}{\sqrt{1+a^2}},x_0\sin\theta\sin\varphi,x_0\sin\theta\cos\varphi,x_0\cos\theta\right)(u,v)\\
=&\frac{x_0}{\sqrt{1+a^2}}(u,v)\left(a,\sqrt{1+a^2}\sin\theta\sin\varphi,\sqrt{1+a^2}\sin\theta\cos\varphi,\sqrt{1+a^2}\cos\theta\right)(u,v)\\
=&\frac{x_0}{\sqrt{1+a^2}}(u,v)\bnu(u,v).
\end{aligned}
\end{equation}
Since $\bx$, $\overline{\bnu}_1$, $\overline{\bnu}_2$ are linearly independent  on the hyperplane pseudo-orthogonal to $\bnu$, there exist $\bnu_1,\bnu_2:U\rightarrow S_1^3$ such that  $\bx$, $ \bnu_1 $, $\bnu_2 $ are pseudo-orthogonal and $\bx\wedge \bnu_1\wedge\bnu_2$ is parallel to $\bnu$. It follows that    $\left\{\bx,\bnu_1,\bnu_2,\bnu\right\}$  is a moving frame along $\bx$ and $\langle\bx,\bnu_1\rangle(u,v)=\langle\bx,\bnu_2\rangle(u,v)=0$.
Hence, there exist smooth functions $\alpha,\beta:U\rightarrow\R$ such that $(\bx\wedge\bx_u\wedge\bx_v)(u,v)=\alpha(u,v)\bnu_1(u,v)+\beta(u,v)\bnu_2(u,v)$, which means that $(\bx,\bnu_1,\bnu_2)$ is a hyperbolic generalized framed surface.
\end{proof}
\end{theorem}
\begin{remark} \rm
If we directly find $\bnu_1,\bnu_2$ in the proof of Theorem \ref{th7}, for linearly independent vectors $\bx,\overline{\bnu}_1,\overline{\bnu}_2$, by Schmidt orthogonalization, we get non-zero vectors $\widetilde{\bnu}_1=\overline{\bnu}_1-k\bx$ satisfying $\langle\widetilde{\bnu}_1,\bx\rangle=0$ and $\widetilde{\bnu}_2=\overline{\bnu}_2-k_1\bx-k_2\widetilde{\bnu}_1$ satisfying $\langle\widetilde{\bnu}_2,\bx\rangle=\langle\widetilde{\bnu}_2,\widetilde{\bnu}_1\rangle=0$, where $k,k_1,k_2:U\rightarrow \R$ are smooth functions.
It can be derived through calculation that
$\vert\widetilde{\bnu}_1 \vert^2= 1+k^2>0 $ and $\vert\widetilde{\bnu}_2 \vert^2= 1+k_1^2/(1+k^2) >0$.
Furthermore, we have
\begin{equation}\notag
\begin{aligned}
 \bnu_1(u,v)=\frac{\widetilde{\bnu}_1}{\vert\widetilde{\bnu}_1\vert}(u,v) &=  \frac{(-kx_0,-kx_1+\cos\theta\sin\varphi,-kx_2+\cos\theta\cos\varphi,-kx_3-\sin\theta)}{\sqrt{1+k^2 }}(u,v), \\
 \bnu_2(u,v)=\frac{\widetilde{\bnu}_2}{\vert\widetilde{\bnu}_2\vert}(u,v) &=  \frac{\sqrt{1+k^2}}{\sqrt{1+k^2+k_1^2}}(u,v)(-k_1x_0+kk_2x_0,\cos\varphi-k_1x_1+kk_2x_1-k_2\cos\theta\sin\varphi,\\
&~~~-\sin\varphi-k_1x_2+kk_2x_2-k_2\cos\theta\cos\varphi,-k_1x_3+kk_2x_3+k_2\sin\theta)(u,v),
\end{aligned}
\end{equation}
where 	\begin{equation}\notag
\begin{aligned}
k(u,v)&=(-x_1\cos\theta\sin\varphi-x_2\cos\theta\cos\varphi+x_3\sin\theta)(u,v),\\
k_1(u,v)&= (-x_1\cos\varphi+x_2\sin\varphi)(u,v),\\
k_2(u,v)&= \frac{-kx_1\cos\varphi+kx_2\sin\varphi}{-(kx_0)^2+(-kx_1+\cos\theta\sin\varphi)^2+(-kx_2+\cos\theta\cos\varphi)^2 +(kx_3+\sin\theta)^2}(u,v).
\end{aligned}
\end{equation}
\end{remark}
\subsection{Singularities of hyperbolic generalized framed base surfaces}
We say a point $\bm{p}$   is a cross cap singular point (respectively, $S_1^{\pm}$ singular point) of $\bx$  if $\bx$ is $ \mathcal{A} $-equivalent to $(u,v)\mapsto(u,v^2,uv)$ (respectively, $(u,v)\mapsto(u,v^2,v(u^2\pm v^2))$) at $\bm{p}$. We  use the following  Lemma   to obtain the conditions for   cross cap and $S_1^{\pm}$ singular points.
\begin{lemma}[\cite{S2010,W1943}]\label{lemma2.17}
Let    $ \bx:( U,\bm{p}  )  \rightarrow  (H^3,\bx(\bm{p}) )  $ be a map germ and $\bm{p}=(u_0,v_0)$ a corank one sigular point of $\bx$. Then,
\begin{itemize}
\item[\rm(1)] $\bx$ at   $\bm{p}$ is $ \mathcal{A} $-equivalent to cross cap 
if and only if $d\varphi(u_0,v_0)\neq\bm{0}$.
\item[\rm(2)] $\bx$ at   $\bm{p}$ is $ \mathcal{A} $-equivalent to $S_1^{+}$ singular point
if and only if $d\varphi(u_0,v_0)=\bm{0}$, $\det {\rm Hess} \;\varphi(u_0,v_0)<0$,   $\xi\bx(\bm{p})$ and $\eta\eta\bx(\bm{p})$ are linearly independent.
\item[\rm(3)] $\bx$ at   $\bm{p}$ is $ \mathcal{A} $-equivalent to $S_1^{-}$ singular point
if and only if $d\varphi(u_0,v_0)=\bm{0}$ and $\det {\rm Hess} \;\varphi(u_0,v_0)>0$. 
\end{itemize}
Here, $\varphi(u,v)=\det(\bx,\xi\bx, \eta\bx,\eta\eta\bx)(u_0,v_0)$, $\xi$ is a vector field transverse to the null vector field  $\eta$ and $\det {\rm Hess} \;\varphi(u,v)=(\varphi_{uu}\varphi_{vv}-\varphi^2_{uv})(u,v)$.
\end{lemma}

\begin{proposition}\label{prop2.18}
Let   $(\bx,\bnu_1,\bnu_2): U \rightarrow H^3 \times \Delta_5$  be  a   hyperbolic generalized  framed surface with   basic invariants $a_i,b_i,c_i,e_i,f_i,g_i$, $i=1,2$. If $\bm{p}=(u_0,v_0)$ is a singular point of $\bx$,  that is, $\alpha(u_0,v_0)=\det\begin{pmatrix}
\bb(u_0,v_0) ~\bc(u_0,v_0) 
\end{pmatrix}=0$ and $\beta(u_0,v_0)=\det\begin{pmatrix}
\bc(u_0,v_0) ~\ba(u_0,v_0) 
\end{pmatrix}=0$. Then,
\begin{itemize}
\item[\rm(1)] $\bm{p}$  is a cross cap singular point of $\bx$  if and only if $(a_1,a_2,b_1,b_2)(u_0,v_0)=(0,0,0,0)$,\\ $(c_1,c_2)(u_0,v_0)\neq(0,0)$ and $ \left(\det\begin{pmatrix}
	\bb_u  ~\bc 
\end{pmatrix}\det\begin{pmatrix}
	\ba_v  ~\bc  
\end{pmatrix}-\det\begin{pmatrix}
	\bb_v  ~\bc 
\end{pmatrix}\det\begin{pmatrix}
	\ba_u  ~\bc  
\end{pmatrix}\right)(u_0,v_0)\neq0 $.
	\item[\rm(2)] $\bm{p}$  is a $S_1^+$ singular point of $\bx$   if and only if $(a_1,a_2,b_1,b_2)(u_0,v_0)=(0,0,0,0)$, $(c_1,c_2)(u_0,v_0)\neq(0,0)$, $ \left(\det\begin{pmatrix}
		\bb_u  ~\bc 
	\end{pmatrix}\det\begin{pmatrix}
		\ba_v  ~\bc  
	\end{pmatrix}-\det\begin{pmatrix}
		\bb_v  ~\bc 
	\end{pmatrix}\det\begin{pmatrix}
		\ba_u  ~\bc  
	\end{pmatrix}\right)(u_0,v_0)=0 $, $\det {\rm Hess} \;\varphi(u_0,v_0)<0$   and  $ \left(-c_1\det\begin{pmatrix}
	\ba_v  ~\bc  
\end{pmatrix}+c_2\det\begin{pmatrix}
\ba_u  ~\bc  
\end{pmatrix},c_2\det\begin{pmatrix}
\bb_u  ~\bc 
\end{pmatrix}-c_1\det\begin{pmatrix}
\bb_v  ~\bc 
\end{pmatrix}\right)(u_0,v_0)\neq(0,0) $. 
\item[\rm(3)] $\bm{p}$  is a $S_1^-$ singular point of $\bx$   if and only if $(a_1,a_2,b_1,b_2)(u_0,v_0)=(0,0,0,0)$, $(c_1,c_2)(u_0,v_0)\neq(0,0)$, $ \left(\det\begin{pmatrix}
	\bb_u  ~\bc 
\end{pmatrix}\det\begin{pmatrix}
	\ba_v  ~\bc  
\end{pmatrix}-\det\begin{pmatrix}
	\bb_v  ~\bc 
\end{pmatrix}\det\begin{pmatrix}
	\ba_u  ~\bc  
\end{pmatrix}\right)(u_0,v_0)=0 $ and $ \det {\rm Hess} \;\varphi(u_0,v_0)>0$.
	\end{itemize}
\begin{proof}
If $(a_1,a_2,b_1,b_2)(u_0,v_0)\neq(0,0,0,0)$, it may be assumed that $a_1(u_0,v_0)\neq0$. Then $a_1(u,v)\neq0$ at least locally around $\bm{p}$. There will exist a smooth function $\theta:U\rightarrow\R$ such that $$\cos\theta(u,v)=\frac{b_1(u,v)}{\sqrt{a_1^2(u,v)+b_1^2(u,v)}},\;\sin\theta(u,v)=\frac{a_1(u,v)}{\sqrt{a_1^2(u,v)+b_1^2(u,v)}}.$$ By Proposition \ref{th11}, we can obtain that $\bx$ is a hyperbolic framed base surface at least locally around $\bm{p}$, in which case $\bx$ is not a cross cap or $S^{\pm}_1$ singular point at $\bm{p}$.   Since cross cap and $S^{\pm}_1$ singular points   are corank one, then $  (\bx_u,\bx_v)(u_0,v_0)= (c_1\bnu_3,c_2\bnu_3)(u_0,v_0)\neq( \bm{0}, \bm{0}) ,$  that is, $(c_1,c_2)(u_0,v_0)\neq(0,0)$.
It follows that $(c_1(u,v),c_2(u,v))\neq(0,0)$ at least locally around $\bm{p}$. Since $$c_2(u,v)\bx_u(u,v)-c_1(u,v)\bx_v(u,v)=-\beta(u,v)\bnu_1(u,v)+\alpha(u,v)\bnu_2(u,v),$$
we can take the null vector field  $\eta(u,v)=c_2(u,v)\partial/\partial u-c_1(u,v)\partial/\partial v$. A vector field transverse to $\eta$ can be given by $\xi(u,v)=c_1(u,v)\partial/\partial u+c_2(u,v)\partial/\partial v.$ 	By  calculating, we obtain 
\begin{equation}\notag
	\begin{aligned}
		\xi\bx(u,v) 
		&=(a_1c_1+a_2c_2)(u,v) \bnu_1(u,v) +(b_1c_1+b_2c_2)(u,v) \bnu_2(u,v)  +(c_1^2+c_2^2)(u,v) \bnu_3(u,v) ,\\
		\eta\eta\bx(u,v)  
		&=(\alpha^2+\beta^2)(u,v)\bx(u,v)+((c_1\beta_v-c_2\beta_u)+\alpha(c_1e_2-c_2e_1))(u,v) \bnu_1(u,v) \\&~~~+((c_2\alpha_u-c_1\alpha_v)+\beta(c_1e_2-c_2e_1))(u,v) \bnu_2(u,v) \\&~~~+(\beta(c_1f_2-c_2f_1)+\alpha(c_2g_1-c_1g_2))(u,v)\bnu_3(u,v).
	\end{aligned}
\end{equation}
Denote \begin{equation}\notag
\begin{aligned}
	\varphi(u,v)& =\det(\bx,\xi\bx, \eta\bx,\eta\eta\bx)(u ,v ) \\&=\det\begin{pmatrix}
		a_1c_1+a_2c_2&-\beta& c_1\beta_v-c_2\beta_u +\alpha(c_1e_2-c_2e_1)\\
		b_1 c_1+b_2c_2&\alpha& c_2\alpha_u-c_1\alpha_v +\beta(c_1e_2-c_2e_1)\\c_1^2+c_2^2&0&\beta(c_1f_2-c_2f_1)+\alpha(c_2g_1-c_1g_2)
	\end{pmatrix} (u,v). 	\end{aligned}	\end{equation}
Then \begin{equation}\notag
	\begin{aligned}
		\varphi_u(u_0,v_0)& =(c_1(c_1^2+c_2^2)(\alpha_v\beta_u-\alpha_u\beta_v))(u_0,v_0),\\
		\varphi_v(u_0,v_0)& =(c_2(c_1^2+c_2^2)(\alpha_v\beta_u-\alpha_u\beta_v))(u_0,v_0).
	\end{aligned}
\end{equation}
It follows that $d\varphi(u_0,v_0)=\bm{0}$ if and only if $$\begin{aligned}
(\alpha_v\beta_u-\alpha_u\beta_v)(u_0,v_0)= \left(\det\begin{pmatrix}
\bb_u  ~\bc 
\end{pmatrix}\det\begin{pmatrix}
\ba_v  ~\bc  
\end{pmatrix}-\det\begin{pmatrix}
\bb_v  ~\bc 
\end{pmatrix}\det\begin{pmatrix}
\ba_u  ~\bc  
\end{pmatrix}\right)(u_0,v_0)=0.	\end{aligned}$$ 

By calculation, we can further obtain
\begin{equation}\notag
	\begin{aligned}
		\varphi_{uu}(u_0,v_0)&=((c_1^2+c_2^2)(2c_1(\beta_u\alpha_{uv}-\alpha_u\beta_{uv})+c_1(\beta_{uu}\alpha_v-\alpha_{uu}\beta_v)+c_2(\alpha_u\beta_{uu}-\beta_u\alpha_{uu})\\&~~~~-2(c_1e_2-c_2e_1)(\alpha_u^2+\beta_u^2)))(u_0,v_0),\\
		\varphi_{vv}(u_0,v_0)&=((c_1^2+c_2^2)(2c_1(\alpha_v\beta_{uv}-\beta_v\alpha_{uv})+c_1(\beta_{v}\alpha_{vv}-\alpha_{v}\beta_{vv})+c_2(\alpha_{vv}\beta_{u}-\alpha_u\beta_{vv})\\&~~~~-2(c_1e_2-c_2e_1)(\alpha_v^2+\beta_v^2)))(u_0,v_0),\\
		\varphi_{uv}(u_0,v_0)&=((c_1^2+c_2^2)(c_1(\alpha_{vv}\beta_u-\alpha_u\beta_{vv})+c_2(\alpha_{ v}\beta_{uu}-\alpha_{uu}\beta_{ v})\\&~~~~-2(c_1e_2-c_2e_1)(\alpha_u\alpha_v +\beta_u\beta_v )))(u_0,v_0) 
	\end{aligned}
\end{equation}
and \begin{equation} \notag
	\det {\rm Hess} \;\varphi (u_0,v_0)=(\varphi_{uu}\varphi_{vv}-\varphi_{uv}^2)(u_0,v_0).	 
\end{equation}
Since $  \xi\bx(u_0,v_0) 
=   (c_1^2+c_2^2)(u_0,v_0) \bnu_3(u_0,v_0) $ and 
$$\begin{aligned}
\eta\eta\bx(u_0,v_0)  
= &\left(-c_1\det\begin{pmatrix}
	\ba_v  ~\bc  
\end{pmatrix}+c_2\det\begin{pmatrix}
\ba_u  ~\bc  
\end{pmatrix}\right) (u_0,v_0) \bnu_1(u_0,v_0) \\ &+\left(c_2\det\begin{pmatrix}
\bb_u  ~\bc 
\end{pmatrix}-c_1\det\begin{pmatrix}
\bb_v  ~\bc 
\end{pmatrix}\right) (u_0,v_0) \bnu_2(u_0,v_0) , 	\end{aligned} $$ $ \xi\bx(u_0,v_0) $  and $ \eta\eta\bx(u_0,v_0)$ are linearly independent at $\bm{p}=(u_0,v_0)$ if and only if $$\left(-c_1\det\begin{pmatrix}
\ba_v  ~\bc  
\end{pmatrix}+c_2\det\begin{pmatrix}
\ba_u  ~\bc  
\end{pmatrix},c_2\det\begin{pmatrix}
\bb_u  ~\bc 
\end{pmatrix}-c_1\det\begin{pmatrix}
\bb_v  ~\bc 
\end{pmatrix}\right)(u_0,v_0)\neq(0,0).$$  Using the criteria for cross cap and $S_1^{\pm}$  singular points in Lemma \ref{lemma2.17}, we   have the assertions $(1)$, $(2)$, $(3)$.  			
\end{proof}
\end{proposition}
\subsection{Examples}
 Several examples are provided as applications of hyperbolic generalized framed surfaces.
\begin{example}[Corank one singularities]
	Let $\bx: (\R^2,(0,0))\rightarrow (H^3,\bx(0,0))$ be given by $$\bx(u,v)=\left(\sqrt{u^2+f^2(u,v)+g^2(u,v)+1},  u, f(u,v),  g(u,v)\right),$$ where $f,g:U\rightarrow\R$ are smooth functions with $f_v(0,0)=g_v(0,0)=0$.
	By a direct calculation, we have 
	\begin{equation}\notag
		\begin{aligned}
			(\bx\wedge\bx_u\wedge\bx_v)(u,v) = \bigg( &(g_vf-f_vg)(u,v)+u(f_vg_u- f_ug_v)(u,v),\\
			&\frac{(1+u^2)(f_v g_u -f_u g_v)(u,v)+u(g_v f - f_vg)(u,v) }{\sqrt{u^2+(f^2 +g^2)(u,v) +1}},\\
			&\frac{(g_v +  g_vf^2-f_vgf)(u,v) +u(ff_v g_u -f f_u g_v)(u,v)}{\sqrt{u^2+(f^2 +g^2)(u,v) +1}},\\
			&\frac{(-f_v +  g_vfg-f_vg^2)(u,v) +u(gf_v g_u -g f_u g_v)(u,v)}{\sqrt{u^2+(f^2 +g^2)(u,v) +1}} 
			\bigg).
		\end{aligned}
	\end{equation}
	There exist unit spacelike vectors 
	\begin{equation}\notag
		\bnu_1(u,v)= \frac{\sqrt{(g(u,v)-ug_u(u,v))^2+g^2_u(u,v)+1}}{\sqrt{p(u,v)}}(\overline{\bnu}_1-k \overline{\bnu}_2)(u,v)
	\end{equation}and
	\begin{equation}\notag
		\bnu_2(u,v)=\frac{(ug_u(u,v)-g(u,v))\bx(u,v)+(0,g_u(u,v),0,-1)}{\sqrt{(g(u,v)-ug_u(u,v))^2+g^2_u(u,v)+1}}
	\end{equation}
	such that $\langle\bx,\bnu_1\rangle(u,v)=\langle\bx,\bnu_2\rangle(u,v)=\langle\bnu_1,\bnu_2\rangle(u,v)=0$ and  $$\begin{aligned}
		&(\bx\wedge\bx_u\wedge\bx_v)(u,v)\\=& ~\frac{  g_v(u,v)\sqrt{p(u,v)}}{\sqrt{u^2+(f^2 +g^2)(u,v) +1}\sqrt{(g(u,v)-ug_u(u,v))^2+g^2_u(u,v)+1}}\bnu_1(u,v)\\&+\frac{ q(u,v)}{\sqrt{u^2+(f^2 +g^2)(u,v) +1}\sqrt{(g(u,v)-ug_u(u,v))^2+g^2_u(u,v)+1}}\bnu_2(u,v), \end{aligned}$$
	where \begin{equation}\notag
		\begin{aligned}
			\overline{\bnu}_1(u,v)&=\Big((f(u,v)-uf_u(u,v))\sqrt{u^2+(f^2 +g^2)(u,v) +1},uf(u,v)-(1+u^2)f_u(u,v),\\&
			~~~~~1+f^2(u,v)-uf(u,v)f_u(u,v),f(u,v)g(u,v)-ug(u,v)f_u(u,v)\Big),\\
			\overline{\bnu}_2(u,v)	&=(ug_u(u,v)-g(u,v))\bx(u,v)+(0,g_u(u,v),0,-1),\\
			k(u,v)	&=\frac{-(gf+g_uf_u)(u,v)+u(g_uf+gf_u)(u,v)-u^2f_u(u,v)g_u(u,v)}{ {(g(u,v)-ug_u(u,v))^2+g^2_u(u,v)+1}},\\
			p(u,v)&=(f(u,v)-uf_u(u,v))^2+(g(u,v)-ug_u(u,v))^2+((f_ug-g_uf)^2+f^2_u+g_u^2)(u,v)+1,\\
			q(u,v)&=g_v(u,v)(-(gf+g_uf_u)(u,v)+u(g_uf+gf_u)(u,v)-u^2f_u(u,v)g_u(u,v))\\&~~+f_v(u,v)((g(u,v)-ug_u(u,v))^2+g^2_u(u,v)+1).		
		\end{aligned}
	\end{equation}
	Thus, $(\bx,\bnu_1,\bnu_2)$ is a hyperbolic generalized framed surface with 
	$$\begin{aligned}
		\alpha(u,v)&=\frac{  g_v(u,v)\sqrt{p(u,v)}}{\sqrt{u^2+(f^2 +g^2)(u,v) +1}\sqrt{(g(u,v)-ug_u(u,v))^2+g^2_u(u,v)+1}},\\ \beta(u,v)&=\frac{ q(u,v)}{\sqrt{u^2+(f^2 +g^2)(u,v) +1}\sqrt{(g(u,v)-ug_u(u,v))^2+g^2_u(u,v)+1}}.\end{aligned}$$
	
\end{example}
\begin{example}[A cross cap]\label{ex3}
	Let $\bx: (\R^2,(0,0))\rightarrow (H^3,\bx(0,0))$ be given by $$\bx(u,v)=\left(\sqrt{u^2+v^4+u^2v^2+1},  u, v^2,  uv\right).$$
	By a direct calculation, we have 
	\begin{equation}\notag
		(\bx\wedge\bx_u\wedge\bx_v)(u,v)=\frac{\left(uv^2\sqrt{u^2+v^4+u^2v^2+1},v^2 (u^2+2),u(v^4+1),u^2 v^3-2v\right)}{\sqrt{u^2+v^4+u^2v^2+1}}.
	\end{equation}
	Denote \begin{equation}\notag
		\bnu(u,v)=\left(  \frac{u\sqrt{v^2+1}}{ \sqrt{v^4+1}}, \frac{ \sqrt{u^2+v^4+u^2v^2+1}}{\sqrt{v^2+1}\sqrt{v^4+1}},0, \frac{v\sqrt{u^2+v^4+u^2v^2+1}}{\sqrt{v^2+1}\sqrt{v^4+1}}\right).
	\end{equation}
	Then we obtain $\langle\bx,\bnu\rangle(u,v)=\langle\bx\wedge\bx_u\wedge\bx_v,\bnu\rangle(u,v)=0$. By Theorem  \ref{th7}, we conclude that $\bx$ is a hyperbolic generalized framed base surface.
	Furthermore, there exist unit spacelike vectors
	$$\bnu_1(u,v)=\left( \frac{v^2\sqrt{u^2+v^4+u^2v^2+1}}{\sqrt{v^4+1}},\frac{u  v^2}{\sqrt{v^4+1}},\sqrt{v^4+1},\frac{u  v^3}{\sqrt{v^4+1}}\right)$$ and $$\bnu_2(u,v)=\left(0,\frac{  v }{\sqrt{v^2+1}},0,-\frac{ 1 }{\sqrt{v^2+1}}\right)$$ such that $\langle\bx,\bnu_1\rangle(u,v)=\langle\bx,\bnu_2\rangle(u,v)=\langle\bnu_1,\bnu_2\rangle(u,v)=0$ and  $$(\bx\wedge\bx_u\wedge\bx_v)(u,v)= \frac{ u\sqrt{v^4+1}}{\sqrt{u^2+v^4+u^2v^2+1}}\bnu_1(u,v)+\frac{ 2v\sqrt{v^2+1}}{\sqrt{u^2+v^4+u^2v^2+1}}\bnu_2(u,v). $$ Thus, $(\bx,\bnu_1,\bnu_2)$ is a hyperbolic generalized framed surface with 
	$$\alpha(u,v)=\frac{ u\sqrt{v^4+1}}{\sqrt{u^2+v^4+u^2v^2+1}},\; \beta(u,v)=\frac{ 2v\sqrt{v^2+1}}{\sqrt{u^2+v^4+u^2v^2+1}}$$
	and the  basic invariants 
	\begin{equation}\notag
		\begin{aligned}
		\begin{pmatrix}
			a_1(u,v)&b_1(u,v)&c_1(u,v)\\
			a_2(u,v)&b_2(u,v)&c_2(u,v)
		\end{pmatrix}&= 
		\begin{pmatrix}
			0&0&\frac{\sqrt{v^4+1}\sqrt{v^2+1} }{\sqrt{u^2+v^4+u^2v^2+1}}\\
			\frac{2v}{\sqrt{v^4+1}}&  -\frac{u}{\sqrt{v^2+1}}&\frac{ uv(1-v^4-2v^2) }{\sqrt{u^2+v^4+u^2v^2+1}\sqrt{v^4+1}\sqrt{v^2+1}}
		\end{pmatrix}, \\
		\begin{pmatrix}
			e_1(u,v)&f_1(u,v)&g_1(u,v)\\
			e_2(u,v)&f_2(u,v)&g_2(u,v)
		\end{pmatrix}&= 
		\begin{pmatrix}
			0& \frac{v^2\sqrt{v^2+1}}{\sqrt{u^2+v^4+u^2v^2+1}}&0\\
			-\frac{uv^2}{\sqrt{v^4+1}\sqrt{v^2+1}}&\frac{ uv^3(1-v^4-2v^2) }{(v^4+1)\sqrt{u^2+v^4+u^2v^2+1} \sqrt{v^2+1}}& \frac{\sqrt{u^2+v^4+u^2v^2+1}}{(v^2+1)\sqrt{v^4+1}}
		\end{pmatrix}.
			\end{aligned}
	\end{equation}
\end{example}
\begin{example}[A hyperbolic ruled surface] 
	Let $\bgamma:(-\pi,\pi]\rightarrow H^3$,  $$\bgamma(u)=\left(\frac{13}{5}, \frac{9\cos u-3\cos 3u}{5}, \frac{9\sin u-3\sin 3u}{5},\frac{6\sqrt{3}\cos u}{5}\right)$$ be a smooth curve in $H^3$
	and $\bm{\delta}:(-\pi,\pi]\rightarrow S_1^3$,
	$$\bm{\delta}(u)=\frac{\left(-156\sin u, -97\sin 2u+18\sin 4u,-50\sin^2u-144\sin^4u+25,-36\sqrt{3}\sin 2u \right)}{5\sqrt{144\sin^2u+25}}.$$
	Then  \begin{equation}\notag
		\begin{aligned}
			\bx(u,v)&=\cosh v\bgamma(u)+\sinh v\bm{\delta}(u)\\&= \Bigg( \frac{13\cosh v}{5}-\frac{156\sin u\sinh v}{5\sqrt{144\sin^2u+25}},
			\frac{(9\cos u-3\cos 3u)\cosh v}{5} -\frac{(97\sin 2u-18\sin 4u)\sinh v }{5\sqrt{144\sin^2u+25}},\\
			&~~~~ \frac{(9\sin u-3\sin 3u)\cosh v}{5} -\frac{(50\sin^2u+144\sin^4u-25) \sinh v }{5\sqrt{144\sin^2u+25}}, \\			
			&~~~~\frac{6\sqrt{3}\cos u\cosh v}{5}-\frac{36\sqrt{3}\sin 2u\sinh v }{5\sqrt{144\sin^2u+25}} \Bigg) 
		\end{aligned}
	\end{equation}
	is a hyperbolic ruled surface generated  by $\bgamma$ and $\bm{\delta}$,
	where $(u,v)\in(-\pi,\pi]\times\R$.
	By a direct calculation, we have 
	\begin{equation}\notag
		\begin{aligned}
			(\bx\wedge\bx_u\wedge\bx_v)(u,v)= &~ \frac{-12\sqrt{3}\sin u\cosh v+\sinh v\sqrt{432\sin^2u+75}}{5}  \bnu_1(u,v)\\
			&+\frac{65\sinh v}{\sqrt{144\sin^2u+25}}\bnu_2(u,v),
		\end{aligned}
	\end{equation}
	where
	$$\begin{aligned}\bnu_1(u,v)&=\frac{\left(24\cos u,  13\cos 2u,13\sin 2u ,13\sqrt{3}  \right)}{2\sqrt{144\sin^2u+25}},\\   \bnu_2(u,v)&=\left(0, -\frac{\sqrt{3}}{2}\cos2u,-\frac{\sqrt{3}}{2}\sin 2u, \frac{1}{2}\right)\end{aligned}$$ and $\langle\bx,\bnu_1\rangle(u,v)=\langle\bx,\bnu_2\rangle(u,v)=\langle\bnu_1,\bnu_2\rangle(u,v)=0$.    Thus, $(\bx,\bnu_1,\bnu_2)$ is a hyperbolic generalized framed surface with 
	$$\alpha(u,v)=\frac{-12\sqrt{3}\sin u\cosh v+\sinh v\sqrt{432\sin^2u+75}}{5},\; \beta(u,v)=\frac{65\sinh v}{\sqrt{144\sin^2u+25}}$$
	and the  basic invariants 
	\begin{equation}\notag
		\begin{aligned}
			\begin{pmatrix}
				a_1(u,v)&b_1(u,v)&c_1(u,v)\\
				a_2(u,v)&b_2(u,v)&c_2(u,v)
			\end{pmatrix}&=
			\begin{pmatrix}
				-\frac{65\sinh v}{\sqrt{144\sin^2u+25}}&\frac{-12\sqrt{3}\sin u\cosh v+\sinh v\sqrt{432\sin^2u+75}}{5}&0\\
				0&0&1
			\end{pmatrix},\\
			\begin{pmatrix}
				e_1(u,v)&f_1(u,v)&g_1(u,v)\\
				e_2(u,v)&f_2(u,v)&g_2(u,v)
			\end{pmatrix}&=
			\begin{pmatrix}
				0& \frac{65\cosh v}{144\sin^2u+25}&\frac{12\sqrt{3}\sin u\sinh v-\cosh v\sqrt{432\sin^2u+75}}{5}\\
				0&0&0
			\end{pmatrix}.
		\end{aligned}
	\end{equation}
	Since $\alpha(0,0)=\beta(0,0)=0$,   $(a_1,a_2,b_1,b_2)(0,0)=(0,0,0,0)$, $(c_1,c_2)(0,0)=(0,1)\neq(0,0)$   and $$ 
	\left(\det\begin{pmatrix}
		\bb_u  ~\bc 
	\end{pmatrix}\det\begin{pmatrix}
		\ba_v  ~\bc  
	\end{pmatrix}-\det\begin{pmatrix}
		\bb_v  ~\bc 
	\end{pmatrix}\det\begin{pmatrix}
		\ba_u  ~\bc  
	\end{pmatrix}\right)(0,0) 
	= 12\sqrt{3}\neq0, $$  we can obtain that $\bx$     is $ \mathcal{A} $-equivalent to cross cap at $(0,0)$ by Proposition \ref{prop2.18}. A similar discussion leads to the conclusion that $\bx$     is also $ \mathcal{A} $-equivalent to cross cap at $(\pi,0)$,  see Figure \ref{figure1}. 
	\begin{figure} 
		\centering
		\includegraphics[width = 7.5 cm]{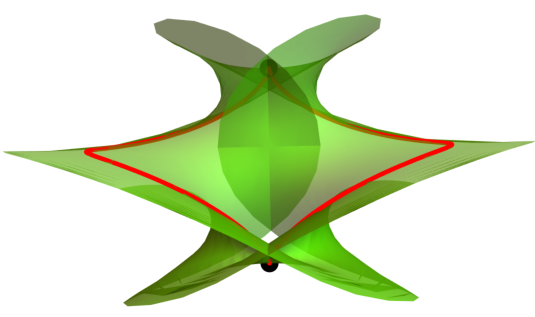}
		\caption{A curve $\bgamma$ and a hyperbolic  ruled surface  $\bx$ projected to Poincar\'{e} 3-disc. The two black points are  cross cap singularities of $\bx$.}
		\label{figure1}
	\end{figure}
	 
\end{example}
\begin{example}[A hyperbolic ruled surface]\label{ex0}
Let $\bgamma:(-\pi,\pi]\rightarrow H^3$,  $$\bgamma(u)=\left(\frac{13}{5},-\frac{3}{5}\cos 3u+\frac{9}{5}\cos u,-\frac{3}{5}\sin 3u+\frac{9}{5}\sin u,\frac{6\sqrt{3}}{5}\cos u\right)$$ be a smooth curve in $H^3$
and $\bm{\delta}:(-\pi,\pi]\rightarrow S_1^3$,
$$\bm{\delta}(u)=\left(0, -\frac{\sqrt{3}}{2}\cos2u,-\frac{\sqrt{3}}{2}\sin 2u, \frac{1}{2}\right).$$
Then  \begin{equation}\notag
\begin{aligned}
\bx(u,v)&=\cosh v\bgamma(u)+\sinh v\bm{\delta}(u)\\&= \Bigg( \frac{13}{5}\cosh v,\left(-\frac{3}{5}\cos 3u+\frac{9}{5}\cos u\right)\cosh v-\frac{\sqrt{3}}{2}\cos 2u\sinh v,\\
&~~~\left(-\frac{3}{5}\sin 3u+\frac{9}{5}\sin u\right)\cosh v-\frac{\sqrt{3}}{2}\sin 2u\sinh v, 
\frac{6\sqrt{3}}{5}\cos u\cosh v+\frac{1}{2}\sinh v \Bigg) 
\end{aligned}
\end{equation}
is a hyperbolic ruled surface generated  by $\bgamma$ and $\bm{\delta}$,
where $(u,v)\in(-\pi,\pi]\times\R$.
By a direct calculation, we have 
\begin{equation}\notag
(\bx\wedge\bx_u\wedge\bx_v)(u,v)=\left( \frac{12\sqrt{3}\cos u\sinh v}{5},\frac{13\sqrt{3}\cos 2u\sinh v}{10},\frac{13\sqrt{3}\sin 2u\sinh v}{10},\frac{39\sinh v}{10}\right).
\end{equation}
	\begin{figure} 
	\centering
	\includegraphics[width = 5.5 cm]{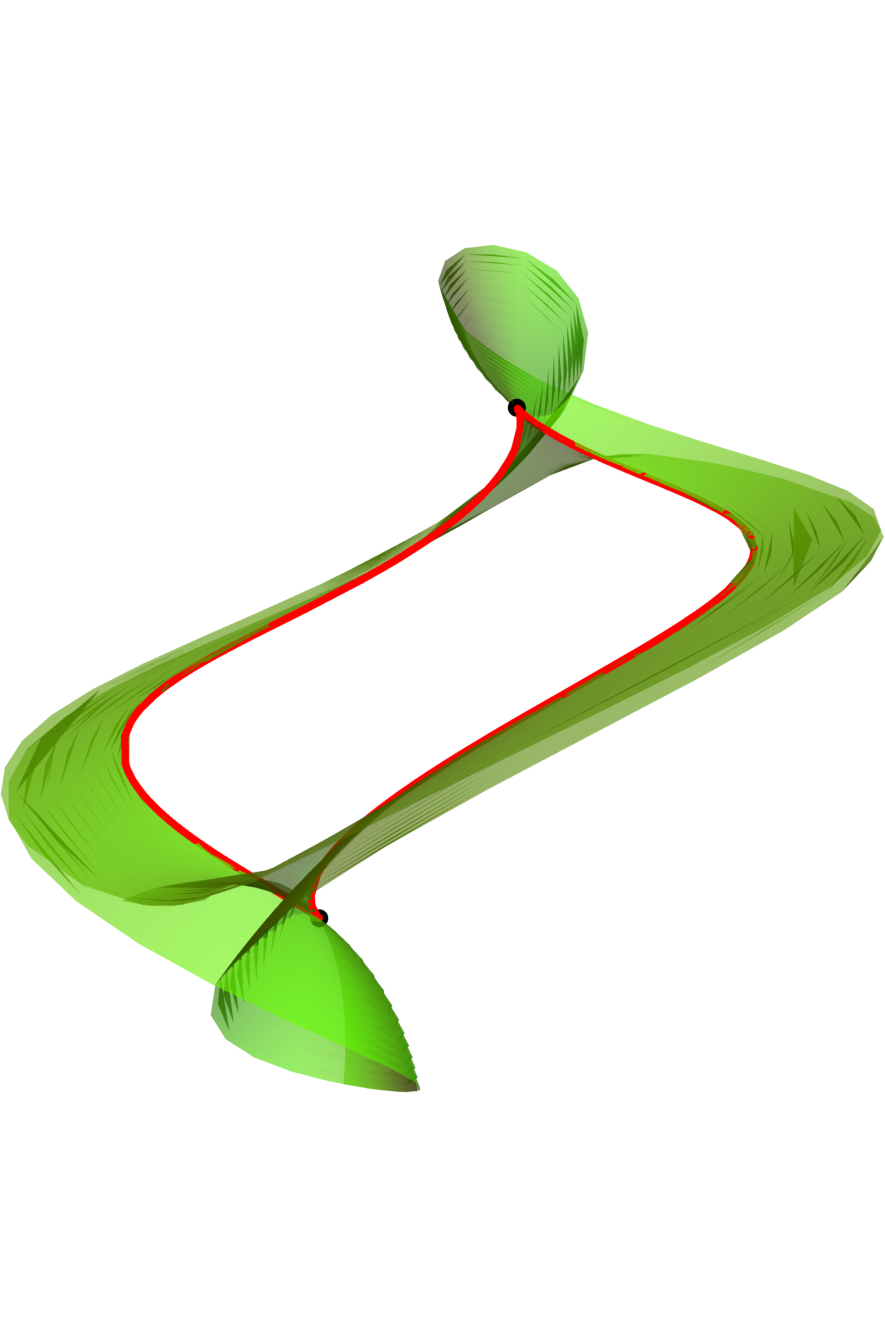}
	\caption{A curve $\bgamma$ and a hyperbolic  ruled surface  $\bx$ projected to Poincar\'{e} 3-disc. The two black points are  swallowtail singularities of $\bx$. The other points on  $\bgamma$ are  cuspidal edge singularities of $\bx$.}
	\label{figure3}
\end{figure}
Denote
$$ \bnu_1(u,v)=\frac{\left(24\cos u,  13\cos 2u,13\sin 2u ,13\sqrt{3}  \right)}{2\sqrt{144\sin^2u+25}} $$
  and  $$ 
  \bnu_2(u,v)= \frac{\left(-156\sin u, -97\sin 2u+18\sin 4u,-50\sin^2u-144\sin^4u+25,-36\sqrt{3}\sin 2u \right)}{5\sqrt{144\sin^2u+25}}. $$  Then we get $\langle\bx,\bnu_1\rangle(u,v)=\langle\bx,\bnu_2\rangle(u,v)=\langle\bnu_1,\bnu_2\rangle(u,v)=0$ and  $$(\bx\wedge\bx_u\wedge\bx_v)(u,v)= \frac{\sqrt{432\sin^2u+75}\sinh v}{5}\bnu_1(u,v) .$$  Thus, $(\bx,\bnu_1,\bnu_2)$ is a hyperbolic generalized framed surface with 
$$\alpha(u,v)=\frac{\sqrt{432\sin^2u+75}\sinh v}{5},\; \beta(u,v)=0$$
and the  basic invariants 
\begin{equation}\notag
	\begin{aligned}
\begin{pmatrix}
a_1(u,v)&b_1(u,v)&c_1(u,v)\\
a_2(u,v)&b_2(u,v)&c_2(u,v)
\end{pmatrix}&=
\begin{pmatrix}
0&-\frac{\sqrt{432\sin^2u+75}\sinh v}{5}&\frac{12\sqrt{3}\sin u}{5} \\
0&0&-1
\end{pmatrix},\\
\begin{pmatrix}
e_1(u,v)&f_1(u,v)&g_1(u,v)\\
e_2(u,v)&f_2(u,v)&g_2(u,v)
\end{pmatrix}&=
\begin{pmatrix}
 \frac{65}{144\sin^2u+25}&0 &-\frac{\sqrt{432\sin^2u+75}\cosh v}{5}\\
0&0&0
\end{pmatrix}.
	\end{aligned}
\end{equation}
By Propositions \ref{prop5} and \ref{prop6}, $(\bx,\bnu_1,\bnu_2)$ is both a hyperbolic framed surface and a one-parameter family of hyperbolic framed curves with respect to $v$.

We have investigated that the singular points of $\bx$ lie on $\bgamma$ in \cite{ZP}. On the singular points of $\bgamma$, $\bx$ has swallowtail singularities, while on the other points, $\bx$ has cuspidal edge singularities, see Figure \ref{figure3}.
\end{example}

\section{Relations between hyperbolic generalized framed surfaces in $H^3$,  generalized framed surfaces in $\R^3$ and lightcone framed surfaces in $\R_1^3$}\label{S4} 
\subsection{Relations between hyperbolic generalized framed surfaces  and  generalized framed surfaces}
Poincar\'{e} 3-disc  $D^3=\left\{(x_1,x_2,x_3)\; | \;x_1^2+x_2^2+x_3^2<1\right\}$ is an open subset of Euclidean  space. It has been known that there is the diffeomorphism $\bm{\pi}$ from $H^3$ to $D^3$,
\begin{equation}\notag
\bm{\pi}(x_1,x_2,x_3,x_4)=\left(\frac{x_2}{x_1+1},\frac{x_3}{x_1+1},\frac{x_4}{x_1+1}\right).	 
\end{equation}
The inverse mapping of $\bm{\pi}$ is 
\begin{equation}\notag
\bm{\pi}^{-1}
(x_1,x_2,x_3)=\frac{({1+x_1^2+x_2^2+x_3^2},{2x_1},{2x_2},{2x_3})}{1-x_1^2-x_2^2-x_3^2}.	 
\end{equation}

If we regard $D^3$ as a subset of $\R^3$ with induced metric, we can obtain the relations between hyperbolic generalized framed surfaces  and generalized  framed surfaces by the diffeomorphism between $H^3$ and $D^3$.
\begin{proposition}\label{th4.1}
Let $(\bx, \bnu_1, \bnu_2) :U\rightarrow H^3 \times \Delta_5$ be a hyperbolic generalized framed surface. Denote $\bx(u,v)=(x_1,x_2,x_3,x_4)(u,v)$, $\bnu_1(u,v)=(y_1,y_2,y_3,y_4)(u,v)$ and $\bnu_2(u,v)=(z_1,z_2,z_3,z_4)(u,v)$. Then   $(\overline{\bx}, \overline{\bnu}_1, \overline{\bnu}_2): U \rightarrow D^3 \times \Delta$  is a generalized framed surface, where
\begin{equation}\notag
\begin{aligned}
\overline{\bx}(u,v)&=\bigg(\frac{x_2}{x_1+1},\frac{x_3}{x_1+1},\frac{x_4}{x_1+1}\bigg)(u,v),\\ 
\overline{\bnu}_1(u,v)&=\dfrac{(p_1,q_1,r_1)}{\sqrt{p_1^2+q_1^2+r_1^2}}(u,v), 
&&\overline{\bnu}_2(u,v)=\dfrac{(p_2,q_2,r_2)}{\sqrt{p_2^2+q_2^2+r_2^2}}(u,v), \\
p_1(u,v)&=(x_2y_1-x_1y_2-y_2)(u,v),
&&q_1(u,v)=(x_3y_1-x_1y_3-y_3)(u,v),\\ 
r_1(u,v)&=(x_4y_1-x_1y_4-y_4)(u,v), 
&&p_2(u,v)=(x_2z_1-x_1z_2-z_2)(u,v),\\ 
q_2(u,v)&=(x_3z_1-x_1z_3-z_3)(u,v), 
&&r_2(u,v)=(x_4z_1-x_1z_4-z_4)(u,v). 
\end{aligned}
\end{equation}
\begin{proof}
Since $(\bx, \bnu_1, \bnu_2) :U\rightarrow H^3 \times \Delta_5$ is a hyperbolic generalized framed surface, there exist smooth functions $\alpha,\beta:U\rightarrow\R$ such that
\begin{equation}\notag
\begin{aligned}
(\bx\wedge\bx_u\wedge\bx_v)(u,v) 
=\big(&-x_2(x_{3u}x_{4v}-x_{4u}x_{3v})+x_3(x_{2u}x_{4v}-x_{4u}x_{2v})-x_4(x_{2u}x_{3v}-x_{3u}x_{2v}),\\
&-x_1(x_{3u}x_{4v}-x_{4u}x_{3v})+x_3(x_{1u}x_{4v}-x_{4u}x_{1v})-x_4(x_{1u}x_{3v}-x_{1v}x_{3u}),\\
&-x_1(x_{2v}x_{4u}-x_{4v}x_{2u})+x_2(x_{1v}x_{4u}-x_{1u}x_{4v})-x_4(x_{1v}x_{2u}-x_{1u}x_{2v}),\\
&-x_1(x_{2u}x_{3v}-x_{3u}x_{2v})+x_2(x_{1u}x_{3v}-x_{1v}x_{3u})-x_3(x_{1u}x_{2v}-x_{1v}x_{2u})\big)(u,v) \\
=(&\alpha y_1+\beta z_1,\alpha y_2+\beta z_2,\alpha y_3+\beta z_3,\alpha y_4+\beta z_4)(u,v).
\end{aligned}
\end{equation}
Then we have
\begin{equation}\label{eq9}
\left\{
\begin{aligned}
&(-x_2(x_{3u}x_{4v}-x_{4u}x_{3v})+x_3(x_{2u}x_{4v}-x_{4u}x_{2v})-x_4(x_{2u}x_{3v}-x_{3u}x_{2v}))(u,v)=(\alpha y_1+\beta z_1)(u,v),\\
&(-x_1(x_{3u}x_{4v}-x_{4u}x_{3v})+x_3(x_{1u}x_{4v}-x_{4u}x_{1v})-x_4(x_{1u}x_{3v}-x_{1v}x_{3u}))(u,v)=(\alpha y_2+\beta z_2)(u,v),\\
&(-x_1(x_{2v}x_{4u}-x_{4v}x_{2u})+x_2(x_{1v}x_{4u}-x_{1u}x_{4v})-x_4(x_{1v}x_{2u}-x_{1u}x_{2v}))(u,v)=(\alpha y_3+\beta z_3)(u,v),\\
&(-x_1(x_{2u}x_{3v}-x_{3u}x_{2v})+x_2(x_{1u}x_{3v}-x_{1v}x_{3u})-x_3(x_{1u}x_{2v}-x_{1v}x_{2u}))(u,v)=(\alpha y_4+\beta z_4)(u,v).
\end{aligned}
\right.
\end{equation}
Let $$ \overline{\bx}(u,v)=\bm{\pi}\circ\bx(u,v)=\bigg(\frac{x_2}{x_1+1},\frac{x_3}{x_1+1},\frac{x_4}{x_1+1}\bigg)(u,v). $$
By calculating, 
\begin{equation}\notag
\begin{aligned}
&(\overline{\bx}_u\times\overline{\bx}_v)(u,v)\\
=&~\frac{1}{(x_1+1)^3}(u,v)\Big( x_2(\alpha y_1+\beta z_1)-x_1(\alpha y_2+\beta z_2)-(\alpha y_2+\beta z_2),\\
& ~~~~~~~~~~~~~~~~~~~~~x_3(\alpha y_1+\beta z_1)-x_1(\alpha y_3+\beta z_3)-(\alpha y_3+\beta z_3),\\
&~~~~~~~~~~~~~~~~~~~~~ x_4(\alpha y_1+\beta z_1)-x_1(\alpha y_4+\beta z_4)-(\alpha y_4+\beta z_4)\Big)(u,v)\\
=&~\frac{1}{(x_1+1)^3}(u,v)\left(\alpha(p_1,q_1,r_1)+\beta(p_2,q_2,r_2) \right)(u,v)\\
=&~\frac{\alpha\sqrt{p_1^2+q_1^2+r_1^2}}{(x_1+1)^3} (u,v)\dfrac{(p_1,q_1,r_1)}{\sqrt{p_1^2+q_1^2+r_1^2}}(u,v)+\frac{\beta\sqrt{p_2^2+q_2^2+r_2^2}}{(x_1+1)^3}(u,v)\dfrac{(p_2,q_2,r_2)}{\sqrt{p_2^2+q_2^2+r_2^2}}(u,v),
\end{aligned}
\end{equation}
where \begin{equation}\notag
\begin{aligned}
&p_1(u,v)=(x_2y_1-x_1y_2-y_2)(u,v),
&&q_1(u,v)=(x_3y_1-x_1y_3-y_3)(u,v),\\ 
&r_1(u,v)=(x_4y_1-x_1y_4-y_4)(u,v), 
&&p_2(u,v)=(x_2z_1-x_1z_2-z_2)(u,v),\\ 
&q_2(u,v)=(x_3z_1-x_1z_3-z_3)(u,v), 
&&r_2(u,v)=(x_4z_1-x_1z_4-z_4)(u,v). 
\end{aligned}
\end{equation}
Denote $$\overline{\bnu}_1(u,v)=\dfrac{(p_1,q_1,r_1)}{\sqrt{p_1^2+q_1^2+r_1^2}}(u,v),\; \overline{\bnu}_2(u,v)=\dfrac{(p_2,q_2,r_2)}{\sqrt{p_2^2+q_2^2+r_2^2}}(u,v).$$ 
Next we show that $\overline{\bnu}_1,\overline{\bnu}_2:U\rightarrow S^2$   are well-defined. Assume that there exists $(u_0,v_0)\in U$ such that
$(p_1,q_1,r_1)(u_0,v_0)=(0,0,0)$, we have
\begin{equation}\notag
y_2(u_0,v_0)=\frac{x_2y_1}{x_1+1}(u_0,v_0),\;
y_3(u_0,v_0)=\frac{x_3y_1}{x_1+1}(u_0,v_0),\;
y_4(u_0,v_0)=\frac{x_4y_1}{x_1+1}(u_0,v_0).
\end{equation}
Since $(-y_1^2+y_2^2+y_3^2+y_4^2)(u_0,v_0)=1$, we can get $(y_2,y_3,y_4)(u_0,v_0)\neq(0,0,0)$.  It may be assumed that $y_2(u_0,v_0)= ({x_2y_1}/(x_1+1))(u_0,v_0)\neq0$, which means that $y_1(u_0,v_0)\neq0$. 

It follows that \begin{equation}\notag
\begin{aligned}
\langle \bx ,\bnu_1 \rangle(u_0,v_0)&=(-x_1y_1+x_2y_2+x_3y_3+x_4y_4)(u_0,v_0)\\
&=\left(-x_1y_1+\frac{x_2^2y_1}{x_1+1}+\frac{x_3^2y_1}{x_1+1}+\frac{x_4^2y_1}{x_1+1}\right)(u_0,v_0)\\
&=-y_1(u_0,v_0)\neq0.
\end{aligned}
\end{equation}
This contradicts that $\langle\bx,\bnu_1\rangle(u,v)=0$ for all $(u,v)\in U$. Thus $(p_1,q_1,r_1)(u,v)\neq(0,0,0)$. Similarly, it can be shown that $(p_2,q_2,r_2)(u,v)\neq(0,0,0)$ for all $(u,v)\in U$. 
Furthermore, 
we can verify that $ (\overline{\bnu}_1 \cdot\overline{\bnu}_2)(u,v)=0$ for all $(u,v)\in U$. 
Thus  $(\overline{\bx},\overline{\bnu}_1,\overline{\bnu}_2)$  is a generalized framed surface.
\end{proof}
\end{proposition}
\begin{proposition}
Let $(\overline{\bx},\overline{\bnu}_1,\overline{\bnu}_2): U \rightarrow D^3 \times \Delta$ be a generalized framed surface. Denote $\overline{\bx}(u,v)=(x_1,x_2,x_3)(u,v)$, $\overline{\bnu}_1(u,v)=(y_1,y_2,y_3)(u,v)$ and $\overline{\bnu}_2(u,v)=(z_1,z_2,z_3)(u,v)$. Then   $(\bx, \bnu_1, \bnu_2) :U\rightarrow H^3\times \Delta_5$ is a hyperbolic generalized framed surface, where
\begin{equation}\notag
\begin{aligned}
&{\bx}(u,v)=\frac{({1+x_1^2+x_2^2+x_3^2},{2x_1},{2x_2},{2x_3})}{1-x_1^2-x_2^2-x_3^2}(u,v),\\
&\bnu_1(u,v)=\dfrac{(p_1,q_1,r_1,s_1)}{\sqrt{-p_1^2+q_1^2+r_1^2+s_1^2}}(u,v),\\
&\bnu_2(u,v)=\dfrac{(p_2,q_2,r_2,s_2)}{\sqrt{-p_2^2+q_2^2+r_2^2+s_2^2}}(u,v),\\
&p_1(u,v)=2(x_1y_1+x_2y_2+x_3y_3)(u,v),\\
&q_1(u,v)=((1+x_1^2-x_2^2-x_3^2)y_1+2x_1x_2y_2+2x_1x_3y_3)(u,v),\\
&r_1(u,v)=(2x_1x_2y_1+(1-x_1^2+x_2^2-x_3^2)y_2+2x_2x_3y_3)(u,v),\\
&s_1(u,v)=(2x_1x_3y_1+2x_2x_3y_2+(1-x_1^2-x_2^2+x_3^2)y_3)(u,v),\\
&{p}_2(u,v)=2(x_1z_1+x_2z_2+x_3z_3)(u,v),\\
&{q}_2(u,v)=((1+x_1^2-x_2^2-x_3^2)z_1+2x_1x_2z_2+2x_1x_3z_3)(u,v),\\
&{r}_2(u,v)=(2x_1x_2z_1+(1-x_1^2+x_2^2-x_3^2)z_2+2x_2x_3z_3)(u,v),\\
&{s}_2(u,v)=(2x_1x_3z_1+2x_2x_3z_2+(1-x_1^2-x_2^2+x_3^2)z_3)(u,v).
\end{aligned}
\end{equation}
\begin{proof}
Since	$(\overline{\bx},\overline{\bnu}_1,\overline{\bnu}_2): U \rightarrow D^3 \times \Delta$ is a generalized framed surface, there exist smooth functions $\overline{\alpha},\overline{\beta}:U\rightarrow\R$ such that\begin{equation}\notag
\begin{aligned}
(\overline{\bx}_u\times\overline{\bx}_v)(u,v)
&= (x_{2u}x_{3v}-x_{2v}x_{3u},x_{3u}x_{1v}-x_{3v}x_{1u},x_{1u}x_{2v}-x_{1v}x_{2u})(u,v)\\
&= (\overline{\alpha}y_1+\overline{\beta}z_1,\overline{\alpha}y_2+\overline{\beta}z_2,\overline{\alpha}y_3+\overline{\beta}z_3)(u,v).
\end{aligned}
\end{equation}
Then we have
\begin{equation}\notag
\left\{
\begin{aligned}
&(x_{2u}x_{3v}-x_{2v}x_{3u})(u,v)=(\overline{\alpha}y_1+\overline{\beta}z_1)(u,v),\\
&(x_{3u}x_{1v}-x_{3v}x_{1u})(u,v)=(\overline{\alpha}y_2+\overline{\beta}z_2)(u,v),\\
&(x_{1u}x_{2v}-x_{1v}x_{2u})(u,v)=(\overline{\alpha}y_3+\overline{\beta}z_3)(u,v).
\end{aligned}
\right.
\end{equation}
Let $$ {\bx}(u,v)=\bm{\pi}^{-1}\circ\overline{\bx}(u,v)=\frac{({1+x_1^2+x_2^2+x_3^2},{2x_1},{2x_2},{2x_3})}{1-x_1^2-x_2^2-x_3^2}(u,v). $$
By calculating, 
\begin{equation}\notag
\begin{aligned}
&(\bx\wedge\bx_u\wedge\bx_v)(u,v)\\
=&-\frac{4}{(1-x_1^2-x_2^2-x_3^2)^3(u,v)} \Big( 2x_1(\overline{\alpha}y_1+\overline{\beta}z_1)  +2x_2(\overline{\alpha}y_2+\overline{\beta}z_2) 
+2x_3(\overline{\alpha}y_3+\overline{\beta}z_3),\\
&(1+x_1^2-x_2^2-x_3^2)(\overline{\alpha}y_1+\overline{\beta}z_1)+2x_1x_2(\overline{\alpha}y_2+\overline{\beta}z_2)
+2x_1x_3(\overline{\alpha}y_3+\overline{\beta}z_3),\\
&2x_1x_2(\overline{\alpha}y_1+\overline{\beta}z_1)+  (1-x_1^2+x_2^2-x_3^2)(\overline{\alpha}y_2+\overline{\beta}z_2)+ 
2x_2x_3(\overline{\alpha}y_3+\overline{\beta}z_3),\\
&2x_1x_3(\overline{\alpha}y_1+\overline{\beta}z_1)+  2x_2x_3(\overline{\alpha}y_2+\overline{\beta}z_2)+ 
(1-x_1^2-x_2^2+x_3^2)(\overline{\alpha}y_3+\overline{\beta}z_3)\Big)(u,v)\\
=&-\frac{4}{(1-x_1^2-x_2^2-x_3^2)^3(u,v)}  \left(\overline{\alpha}(p_1,q_1,r_1,s_1)+\overline{\beta}(p_2,q_2,r_2,s_2)\right)(u,v)\\
=&-\frac{4\overline{\alpha}\sqrt{-p_1^2+q_1^2+r_1^2+s_1^2}}{(1-x_1^2-x_2^2-x_3^2)^3}(u,v) \dfrac{(p_1,q_1,r_1,s_1)}{\sqrt{-p_1^2+q_1^2+r_1^2+s_1^2}}(u,v)\\
&-\frac{4\overline{\beta}\sqrt{-p_2^2+q_2^2+r_2^2+s_2^2}}{(1-x_1^2-x_2^2-x_3^2)^3}(u,v) \dfrac{(p_2,q_2,r_2,s_2)}{\sqrt{-p_2^2+q_2^2+r_2^2+s_2^2}}(u,v),
\end{aligned}
\end{equation}
where
\begin{equation}\notag
\begin{aligned}
&p_1(u,v)=2(x_1y_1+x_2y_2+x_3y_3)(u,v),\\
&q_1(u,v)=((1+x_1^2-x_2^2-x_3^2)y_1+2x_1x_2y_2+2x_1x_3y_3)(u,v),\\
&r_1(u,v)=(2x_1x_2y_1+(1-x_1^2+x_2^2-x_3^2)y_2+2x_2x_3y_3)(u,v),\\
&s_1(u,v)=(2x_1x_3y_1+2x_2x_3y_2+(1-x_1^2-x_2^2+x_3^2)y_3)(u,v),\\
&p_2(u,v)=2(x_1z_1+x_2z_2+x_3z_3)(u,v),\\
&q_2(u,v)=((1+x_1^2-x_2^2-x_3^2)z_1+2x_1x_2z_2+2x_1x_3z_3)(u,v),\\
&r_2(u,v)=(2x_1x_2z_1+(1-x_1^2+x_2^2-x_3^2)z_2+2x_2x_3z_3)(u,v),\\
&s_2(u,v)=(2x_1x_3z_1+2x_2x_3z_2+(1-x_1^2-x_2^2+x_3^2)z_3)(u,v).
\end{aligned}
\end{equation}
Denote
\begin{equation}\notag
\begin{aligned}
\bnu_1(u,v)=\dfrac{(p_1,q_1,r_1,s_1)}{\sqrt{-p_1^2+q_1^2+r_1^2+s_1^2}}(u,v), \;
\bnu_2(u,v)=\dfrac{(p_2,q_2,r_2,s_2)}{\sqrt{-p_2^2+q_2^2+r_2^2+s_2^2}}(u,v).
\end{aligned}
\end{equation}
It can be verified that $  (-p_1^2+q_1^2+r_1^2+s_1^2)(u,v)>0 $  and  $ (-p_2^2+q_2^2+r_2^2+s_2^2)(u,v)>0. $  Hence, 
$\bnu_1,\bnu_2:U\rightarrow S_1^3$  are well-defined.  Furthermore,
the calculations show that $  \langle \bx ,\bnu_1\rangle(u,v)=\langle \bx ,\bnu_2 \rangle(u,v)=\langle \bnu_1 ,\bnu_2 \rangle(u,v) =0.$ 
Thus $(\bx, \bnu_1, \bnu_2) $ is a hyperbolic generalized framed surface.
\end{proof}
\end{proposition} 
\subsection{Relations between hyperbolic generalized framed base surfaces and  lightcone framed base surfaces}
In order to investigate the relations between hyperbolic generalized framed base  surfaces  in $H^3$ and  lightcone framed base surfaces in $\R_1^3$, we consider the projections $\bm{\pi}_i:H^3\rightarrow\R_1^3$ ($i=2,3,4$),  where $\bm{\pi}_2 (x_1,x_2,x_3,x_4)=(x_1,x_3,x_4)$, $\bm{\pi}_3(x_1,x_2,x_3,x_4)=(x_1,x_2,x_4)$ and $\bm{\pi}_4(x_1,x_2,x_3,x_4)=(x_1,x_2,x_3)$.  
\begin{proposition}
Let $(\bx, \bnu_1, \bnu_2) :U\rightarrow H^3 \times \Delta_5$ be a hyperbolic generalized framed surface. 
If  $\bm{\pi}_i\circ\bnu_3(u,v)$ is spacelike for all $ (u,v)\in U $, then $\widetilde{\bx}=\bm{\pi}_i\circ\bx:U\rightarrow \R_1^3$  is a lightcone framed base surface,  $i=2,3,4$.
\begin{proof}
	We discuss the case when $i=4$, the proof for other cases follows similarly.
	Denote $\bx(u,v)=(x_1,x_2,x_3,x_4)(u,v)$, $\bnu_1(u,v)=(y_1,y_2,y_3,y_4)(u,v)$ and  $\bnu_2(u,v)=(z_1,z_2,z_3,z_4)(u,v)$.
Since $(\bx, \bnu_1, \bnu_2)  $ is a hyperbolic generalized framed surface,   there exist smooth functions $\alpha,\beta:U\rightarrow\R$ such that equation (\ref{eq9}) holds.
Let
$$\widetilde{\bx}(u,v)=\bm{\pi}_4\circ\bx(u,v) =(x_1,x_2,x_3)(u,v).$$
By calculating, 
\begin{equation}\notag
\begin{aligned}
&(\widetilde{\bx}_u\wedge\widetilde{\bx}_v)(u,v)\\
=&~\alpha(u,v)\left(y_4(x_1,x_2,x_3)-x_4(y_1,y_2,y_3)\right)(u,v)+\beta(u,v)\left(z_4(x_1,x_2,x_3)-x_4(z_1,z_2,z_3)\right)(u,v).
\end{aligned}
\end{equation}
Denote $\bnu_3(u,v)=(\bx\wedge\bnu_1\wedge\bnu_2)(u,v)=(w_1,w_2,w_3,w_4)(u,v)$ and $\bm{\pi}_4\circ\bnu_3(u,v)=(w_1,w_2,w_3)(u,v)$. Then we can get $\langle\widetilde{\bx}_u\wedge\widetilde{\bx}_v,\bm{\pi}_4\circ\bnu_3\rangle(u,v)=0$ for all $ (u,v)\in U $. Since $\bm{\pi}_4\circ\bnu_3$ is spacelike, we can take $$\bt(u,v)=\frac{\bm{\pi}_4\circ\bnu_3}{|\bm{\pi}_4\circ\bnu_3|}(u,v).$$ Using Theorem 3.4 in \cite{LPT2025}, $\langle\widetilde{\bx}_u\wedge\widetilde{\bx}_v,\bt\rangle(u,v)=0$ for all $ (u,v)\in U $  implies that $\widetilde{\bx}$ is a lightcone framed base surface.
\end{proof}
\end{proposition}
From $\R_1^3$ to $H^3$, we consider the maps $\bm{p}_i:\R_1^3\rightarrow H^3$ ($i=2,3,4$), where  $$ \bm{p}_2 (a,b,c)=\left(a,\sqrt{-1+a^2-b^2-c^2},b,c\right),$$  $$ \bm{p}_3 (a,b,c)=\left(a,b,\sqrt{-1+a^2-b^2-c^2},c\right),$$ 
$$ \bm{p}_4 (a,b,c)=\left(a,b,c,\sqrt{-1+a^2-b^2-c^2}\right)$$ and $a^2-b^2-c^2>1$.
\begin{proposition}
Let $(\widetilde{\bx}, \bm{\ell}^+, \bm{\ell}^-) :U\rightarrow \R_1^3 \times \Delta_4$ be a lightcone framed surface.  Denote  $\widetilde{\bx}(u,v)=(x_1,x_2,x_3)(u,v)$
If $(x_1^2-x_2^2-x_3^2)(u,v)>1$ for all $ (u,v)\in U $, then $ {\bx}=\bm{p}_i\circ \widetilde{\bx}:U\rightarrow H^3$  is a hyperbolic generalized framed base surface,  $i=2,3,4$.
\begin{proof}
Without loss of generality, we assume 
$i=4$, the other cases can be proved similarly.
 Denote   $\bm{\ell}^+(u,v)=(y_1,y_2,y_3)(u,v)$ and $\bm{\ell}^-(u,v)=(z_1,z_2,z_3)(u,v)$.
Since $(\widetilde{\bx}, \bm{\ell}^+, \bm{\ell}^-)$ is a   lightcone framed surface,      there exist smooth functions $\widetilde{\alpha},\widetilde{\beta}:U\rightarrow\R$ such that  
\begin{equation}\notag
\begin{aligned}
(\widetilde{\bx}_u\wedge\widetilde{\bx}_v)(u,v)&=
(-x_{2u}x_{3v}+x_{2v}x_{3u},x_{3u}x_{1v}-x_{3v}x_{1u},x_{1u}x_{2v}-x_{1v}x_{2u})(u,v)\\
&=(\widetilde{\alpha}y_1+\widetilde{\beta}z_1,\widetilde{\alpha}y_2+\widetilde{\beta}z_2,\widetilde{\alpha}y_3+\widetilde{\beta}z_3)(u,v).
\end{aligned}
\end{equation}

Let
$${\bx}(u,v)=\bm{p}_4\circ\widetilde{\bx}(u,v) =\left(x_1,x_2,x_3,\sqrt{-1+x_1^2-x_2^2-x_3^2}\right)(u,v).$$
Denote $x_4=\sqrt{-1+x_1^2-x_2^2-x_3^2}$. By calculating, 
\begin{equation}\notag
\begin{aligned}
	&(\bx\wedge {\bx}_u\wedge{\bx}_v)(u,v)\\
	=&~\widetilde{\alpha}(u,v)\left(-\frac{1}{x_4}(y_1,y_2,y_3,0)+\frac{x_1y_1-x_2y_2-x_3y_3}{x_4}(x_1,x_2,x_3,x_4)\right)(u,v)\\&+\widetilde{\beta}(u,v)\left(-\frac{1}{x_4}(z_1,z_2,z_3,0)+\frac{x_1z_1-x_2z_2-x_3z_3}{x_4}(x_1,x_2,x_3,x_4)\right)(u,v).
\end{aligned}
\end{equation}
Take $$\bm{\bnu}(u,v)=\frac{\left(-x_4(y_2z_3-y_3z_2),-x_4(y_1z_3-y_3z_1),x_4(y_1z_2-y_2z_1),-\langle\widetilde{\bx},\bm{\ell}^+\wedge\bm{\ell}^-\rangle\right)}{\sqrt{4x_4^2+\left(\langle\widetilde{\bx},\bm{\ell}^+\wedge\bm{\ell}^-\rangle\right)^2}}(u,v).$$ It can be verified that $\langle\bnu,\bnu\rangle(u,v)=1$ and $\langle\bnu,\bx\rangle(u,v)=\langle\bnu,\bx\wedge\bx_u\wedge\bx_v\rangle(u,v)=0$ for all $(u,v)\in U$.  Using Theorem \ref{th7}, we can conclude that $ {\bx}$ is a hyperbolic generalized framed base surface.
\end{proof}
\end{proposition}
\section{Horocyclic surfaces}\label{S5}
Let $\ba_0:I\rightarrow H^3$ and $ \ba_1,
\ba_2:I\rightarrow S_1^3$  be smooth mappings   with  $\langle \ba_i(u),\ba_j(u)\rangle=0$ ($i\neq j$, $i,j\in\{0,1,2\} $) for all $u\in I$. Denote $\ba_3(u)=(\ba_0 \wedge\ba_1 \wedge\ba_2)(u) $, then $\left\{\ba_0,\ba_1,\ba_2,\ba_3\right\}$ is a pseudo-orthogonal frame of $\R_1^4$ and we have the followings.
\begin{equation}\notag
\begin{pmatrix}
\ba'_0(u)\\
\ba'_1(u)\\
\ba'_2(u)\\
\ba'_3(u)	 
\end{pmatrix}=
\begin{pmatrix}
0&h_1(u)&h_2(u)&h_3(u)\\
h_1(u)&0&h_4(u)&h_5(u)\\
h_2(u)&-h_4(u)&0&h_6(u)\\
h_3(u)&-h_5(u)&-h_6(u)&0	 
\end{pmatrix}
\begin{pmatrix}
\ba_0(u)\\
\ba_1(u)\\
\ba_2(u)\\
\ba_3(u)	 
\end{pmatrix},
\end{equation}
where $h_i:I\rightarrow\R$ are smooth functions, $i=1,2,\dots,6$.
\begin{definition}[\cite{IST2010}] \rm Let  $\bx$ be a smooth surface in $H^3$, $\bx$ is a {\it horocyclic surface} if it is locally parameterized by one-parameter families of horocycles around any point, that is, $$ \bx(u,v)=\ba_0(u)+v\ba_1(u)+\frac{v^2}{2}(\ba_0(u)+\ba_2(u)),$$  where $(u,v)\in I\times\R$.	 
\end{definition}
We can prove that horocyclic surfaces are hyperbolic generalized framed base surfaces.
\begin{proposition}\label{th4.2}
Under the above notations,	let $\bx:I\times\R\rightarrow H^3$ be a horocyclic surface given by $$\bx(u,v)=\ba_0(u)+v\ba_1(u)+\frac{v^2}{2}(\ba_0(u)+\ba_2(u)).$$
Then $(\bx, \bnu_1, \bnu_2) $ is a hyperbolic generalized framed surface, where
\begin{equation}\notag
\bnu_1(u,v)=\ba_3(u),\;
\bnu_2(u,v)=-\frac{v^2}{2}\ba_0(u)-v\ba_1(u)+\left(1-\frac{v^2}{2}\right)\ba_2(u) 		 
\end{equation}
with \begin{equation}\notag
\alpha(u,v)=vh_1(u)-h_2(u)-vh_4(u),\;
\beta(u,v)=\left(1+\frac{v^2}{2}\right)h_3(u)+vh_5(u)+\frac{v^2}{2}h_6(u)
\end{equation}
and the basic invariants
\begin{equation}\notag
\begin{aligned}
&a_1(u,v)=\left(1+\frac{v^2}{2}\right)h_3(u)+vh_5(u)+\frac{v^2}{2}h_6(u), &&a_2(u,v)=0,\\
&b_1(u,v)=-vh_1(u)+h_2(u)+vh_4(u), &&b_2(u,v)=0,\\
&c_1(u,v)=\left(\frac{v^2}{2}-1\right)h_1(u)-vh_2(u)-\frac{v^2}{2}h_4(u), &&c_2(u,v)=-1,\\
&e_1(u,v)=\frac{v^2}{2}h_3(u)+vh_5(u)+\left(\frac{v^2}{2}-1\right)h_6(u),  &&e_2(u,v)=0,\\
&f_1(u,v)=vh_3(u)+h_5(u)+vh_6(u), &&f_2(u,v)=0,\\
&g_1(u,v)=-\frac{v^2}{2}h_1(u)+vh_2(u)+\left(1+\frac{v^2}{2}\right)h_4(u), &&g_2(u,v)=1.	 
\end{aligned} 
\end{equation}
\end{proposition}
\begin{corollary}
Since $a_2(u,v)=b_2(u,v)=0$, by Proposition \ref{prop6} (2)(ii), we know that a horocyclic surface is a one-parameter family of hyperbolic framed base curves with respect to $v$.
\end{corollary}

The horocycle surfaces are the analogies of ruled surfaces. The authors in \cite{IST2010} investigated the flatness of horocycle surfaces and conducted a further classification of the surfaces. Using the basic invariants of hyperbolic generalized framed surfaces, we derive the following classification results.
\begin{proposition}\label{prop4.5}
Let $(\bx,\bnu_1,\bnu_2):I\times\R\rightarrow H^3\times\Delta_5$ be a hyperbolic generalized framed surface  with   basic invariants $a_i,b_i,c_i,e_i,f_i,g_i$ $(i=1,2) $ as in Theorem \ref{th4.2}. Then we have the followings.
\begin{itemize}
\item [\rm(1)] $\bx$ is a horo-flat surface if and only if $c_1(u,v)+g_1(u,v)=b_1(u,v)=0$ for all $(u,v)\in U$.
\item [\rm(2)] $\bx$ is a generalized horo-cone if and only if $c_1(u,v)=g_1(u,v)=b_1(u,v)= (v^2+2)a_1(u,v)-v^2e_1(u,v)-2vf_1(u,v)=0$ for all $(u,v)\in U$.
\item [\rm(3)]  $\bx$ is a horo-cone with a single vertex if and only if $c_1(u,v)=g_1(u,v)=b_1(u,v)= (v^2+2)a_1(u,v)-v^2e_1(u,v)-2vf_1(u,v)= f_1(u,v)-v(a_1(u,v)-e_1(u,v))=0$ and $a_1(u,v)-e_1(u,v)\neq0$ for all $(u,v)\in U$.
\item [\rm(4)]   $\bx$ is a horo-cone with two vertices if and only if $c_1(u,v)=g_1(u,v)=b_1(u,v)=(v^2+2)a_1(u,v)-v^2e_1(u,v)-2vf_1(u,v)=0$ and $f_1(u,v)-v(a_1(u,v)-e_1(u,v))=\lambda(a_1(u,v)-e_1(u,v))\neq0$  for all $(u,v)\in U$, where $\lambda\in\R$.
\item [\rm(5)]  $\bx$ is a  conical horosphere if and only if $c_1(u,v)=g_1(u,v)=b_1(u,v)=a_1(u,v)-e_1(u,v)= e_1(u,v)-vf_1(u,v)=0$ and $f_1(u,v)\neq0$  for all $(u,v)\in U$.
\end{itemize}
\begin{proof}
From the results of horo-flat surfaces in \cite{I2009,IST2010}, we obtain $\bx$ is a horo-flat surface if and only if $h_2(u)=h_4(u)-h_1(u)=0$.  $\bx$ is a generalized horo-cone if  $h_1(u)=h_2(u)=h_3(u)=h_4(u)=0$.   A generalized horo-cone with $h_5(u)=0$ and $h_6(u)\neq0$  is called a horo-cone with a single vertex. A horo-cone with two vertices is the generalized horo-cone that satisfies $h_5(u)\neq0$ and $h_5(u)=\lambda h_6(u)$, where $ \lambda\in\R $. If $h_1(u)=h_2(u)=h_3(u)=h_4(u)=h_6(u)=0$ and $h_5(u)\neq0$, $\bx$ is call conical horosphere.

By calculating, we have $c_1(u,v)+g_1(u,v)=h_4(u)-h_1(u)$, $ b_1(u,v)-v(c_1(u,v)+g_1(u,v))=h_2(u) $,  $a_1(u,v)-e_1(u,v)=h_3(u)+h_6(u)$, $f_1(u,v)-v(a_1(u,v)-e_1(u,v))=h_5(u)$ and
$(v^2+2)a_1(u,v)-v^2e_1(u,v)-2vf_1(u,v)=2h_3(u)$. Based on the above definitions and discussions, we can derive conclusions $ (1) $-$ (5) $.
\end{proof}
\end{proposition}

Furthermore, we can consider the singularities of horocycle surfaces. The  cross cap singularities of horocyclic surfaces  have been investigated in \cite{IST2010} . As a corollary of Proposition  \ref{prop2.18}, we obtain the following.
\begin{corollary}\label{cor4.6}
Let $(\bx,\bnu_1,\bnu_2):U\rightarrow H^3\times\Delta_5$ be a hyperbolic generalized framed surface  with  basic invariants $a_i,b_i,c_i,e_i,f_i,g_i$  $(i=1,2)$ defined as Theorem \ref{th4.2}. Suppose that $\bm{{p}}=(u_0,v_0)$ is a singular point of $\bx$, that is, $a_1(u_0,v_0)=b_1(u_0,v_0)=0.$  Then
\begin{itemize}
	\item [\rm(1)] 
	$\bm{{p}}$ is a cross cap singular point of $\bx$ if and only if $( a_{1u}b_{1v}-a_{1v}b_{1u} )(u_0,v_0)\neq0$.
	
	\item[\rm(2)] $\bm{p}$  is a $S_1^+$ singular point of $\bx$   if and only if  $( a_{1u}b_{1v}-a_{1v}b_{1u} )(u_0,v_0)=0$,   $(c_1a_{1v}+a_{1u},c_1b_{1v}+b_{1u})(u_0,v_0)\neq(0,0)$ and $\det {\rm Hess} \;\varphi (u_0,v_0)<0$.
	\item[\rm(3)] $\bm{p}$  is a $S_1^-$ singular point of $\bx$   if and only if $( a_{1u}b_{1v}-a_{1v}b_{1u} )(u_0,v_0)=0$  and $\det {\rm Hess} \;\varphi (u_0,v_0)>0$.
\end{itemize}
 
\end{corollary}

\vspace{0.3cm}

\noindent Donghe Pei\\
School of Mathematics and Statistics \\
Northeast Normal University \\
Changchun 130024 \\
P. R. China\\
{peidh340@nenu.edu.cn} 

\vspace{0.3cm}
\noindent Masatomo Takahashi\\
Muroran Institute of Technology\\
Muroran 0508585\\
Japan\\
{masatomo@muroran-it.ac.jp}

\vspace{0.3cm}
\noindent Anjie Zhou\\
School of Mathematics and Statistics \\
Northeast Normal University \\
Changchun 130024 \\
P. R. China\\
{zhouaj882@nenu.edu.cn}

\end{document}